\documentclass[a4paper,12pt]{article}

\usepackage[vcentering,dvips]{geometry}
\usepackage{amsthm}

\usepackage{amsmath}
\usepackage{amssymb}
\usepackage{varioref}
\usepackage{hyperref}

\newcommand{\ZZ}{\mathbb{Z}}
\newcommand{\NN}{\mathbb{N}}
\newcommand{\QQ}{\mathbb{Q}}
\newcommand{\RR}{\mathbb{R}}
\newcommand{\PP}{\mathbb{P}}
\newcommand{\FF}{\mathbb{F}}
\newcommand{\D}{\mathcal{D}}
\newcommand{\CO}{\mathcal{O}}
\DeclareMathOperator{\minvert}{min_{\text{vert}}}
\DeclareMathOperator{\loc}{Loc}

\DeclareMathOperator{\ord}{ord}
\DeclareMathOperator{\Pol}{Pol}
\DeclareMathOperator{\pdv}{X}

\DeclareMathOperator{\Proj}{\mathbf{Proj}}
\DeclareMathOperator{\spec}{Spec}
\DeclareMathOperator{\Spec}{\mathbf{Spec}}
\DeclareMathOperator{\Hom}{Hom}
\DeclareMathOperator{\tail}{tail}
\DeclareMathOperator{\SF}{SF}

\DeclareMathOperator{\tcadiv}{T-CaDiv}
\DeclareMathOperator{\divisor}{div}
\DeclareMathOperator{\wdiv}{Div}

\DeclareMathOperator{\conv}{conv}
\DeclareMathOperator{\vol}{vol}
\DeclareMathOperator{\inn}{int}
\DeclareMathOperator{\graph}{\Gamma}
\newcommand{\fan}{\Xi}
\DeclareMathOperator{\pr}{pr}
\DeclareMathOperator{\ev}{ev}

\usepackage[utf8]{inputenc}
\usepackage{subfigure}
\usepackage[dvips]{color}
\usepackage{pstricks}
\usepackage{pst-3d}
\usepackage{pst-grad}
\usepackage{pst-3dplot}

\newtheorem{thm}{Theorem}[section]
\newtheorem{lem}[thm]{Lemma}
\newtheorem{prop}[thm]{Proposition}
\newtheorem{cor}[thm]{Corollary}

\theoremstyle{definition}
\newtheorem{defn}[thm]{Definition}
\newtheorem{exmp}[thm]{Example}
\newtheorem{rem}[thm]{Remark}

\title{AG Codes from Polyhedral Divisors}
\author{Nathan Owen Ilten and Hendrik S\"u\ss{}}
\def\address#1#2{\begingroup
\noindent\parbox[t]{7.8cm}{%
\small{\scshape\ignorespaces#1}\par\vskip1ex
\noindent\small{\itshape E-mail address}%
\/: #2\par\vskip4ex}\hfill%
\endgroup}%

\newcommand{\exmpi}{%
 \psset{unit=0.5cm}
 \begin{pspicture}(-3,-1)(3,1)%
\psgrid[gridwidth=0.3pt,griddots=5,subgriddiv=1,gridlabels=5pt](-3,-1)(3,1)
  \psset{linewidth=1pt}%
  \psline{<-]}(-3,0)(0.5,0)
  \psline{[->}(0.5,0)(3,0)
 \end{pspicture}
}

\newcommand{\exmpismooth}{%
 \psset{unit=0.5cm}
 \begin{pspicture}(-3,-1)(3,1)%
\psgrid[gridwidth=0.3pt,griddots=5,subgriddiv=1,gridlabels=5pt](-3,-1)(3,1)
  \psset{linewidth=1pt}%
  \psline{<-*}(3,0)(1,0)
  \psline{-]}(1,0)(0.5,0)
  \psline{[-*}(0.5,0)(0,0)
  \psline{->}(0,0)(-3,0)
 \end{pspicture}
}

\newcommand{\exmpii}{%
 \psset{unit=0.5cm}
 \begin{pspicture}(-3,-1)(3,1)%
\psgrid[gridwidth=0.3pt,griddots=5,subgriddiv=1,gridlabels=5pt](-3,-1)(3,1)
  \psset{linewidth=1pt}%
  \psline{<-]}(3,0)(1,0)
  \psline{[-]}(1,0)(-1,0)
  \psline{[-]}(-1,0)(-2,0)
  \psline{[->}(-2,0)(-3,0)
 \end{pspicture}
}
\newcommand{\exmpiismooth}{%
 \psset{unit=0.5cm}
 \begin{pspicture}(-3,-1)(3,1)%
\psgrid[gridwidth=0.3pt,griddots=5,subgriddiv=1,gridlabels=5pt](-3,-1)(3,1)
  \psset{linewidth=1pt}%
  \psline{<-]}(3,0)(1,0)
  \psline{[-*}(1,0)(0,0)
  \psline{-]}(0,0)(-1,0)
  \psline{[-]}(-1,0)(-2,0)
  \psline{[->}(-2,0)(-3,0)
 \end{pspicture}
}

\newcommand{\exmpih}{%
 \psset{yunit=0.4cm}
 \psset{xunit=0.5cm}
 \begin{pspicture}(-3,-11)(3,1)%
   \psgrid[gridwidth=0.3pt,griddots=5,subgriddiv=1,gridlabels=5pt](-3,-11)(3,1)
  \psset{linewidth=1pt}%
  \psline{<-]}(-3,0)(0.5,0)
  \psline{[->}(0.5,0)(3,0)
  \psset{linewidth=1pt,linecolor=gray}%
  \psline{-}(3,0)(0.5,0)
  \psline{-}(0.5,0)(-2.25,-11)
 \end{pspicture}
}

\newcommand{\exmpiih}{%
 \psset{yunit=0.4cm}
 \psset{xunit=0.5cm}
 \begin{pspicture}(-3,-11)(3,1)%
   \psgrid[gridwidth=0.3pt,griddots=5,subgriddiv=1,gridlabels=5pt](-3,-11)(3,1)
  \psset{linewidth=1pt}%
  \psline{<-]}(3,0)(1,0)
  \psline{[-]}(-1,0)(1,0)
  \psline{[-]}(-1,0)(-2,0)
  \psline{[->}(-2,0)(-3,0)
  \psset{linewidth=1pt,linecolor=gray}%
  \psline{-}(3,0)(1,0)
  \psline{-}(1,0)(-1,-4)
  \psline{-}(-1,-4)(-2,-7)
  \psline{-}(-2,-7)(-3,-11)
 \end{pspicture}
}

\newcommand{\exmpip}{%
 \psset{unit=0.5cm}
 \begin{pspicture}(-1,-2)(5,3)%
   \psgrid[gridwidth=0.3pt,griddots=5,subgriddiv=1,gridlabels=5pt](-1,-2)(5,3)
  \psline{[-]}(0,0)(4,0)
  \psset{linewidth=1pt,linecolor=gray}%
  \psline{-}(0,0)(4,2)
 \end{pspicture}
}

\newcommand{\exmpiip}{%
 \psset{unit=0.5cm}
 \begin{pspicture}(-1,-2)(5,3)%
   \psgrid[gridwidth=0.3pt,griddots=5,subgriddiv=1,gridlabels=5pt](-1,-2)(5,3)
  \psline{[-]}(0,0)(4,0)
  \psset{linewidth=1pt,linecolor=gray}%
  \psline{-}(0,0)(2,2)
  \psline{-}(2,2)(3,1)
  \psline{-}(3,1)(4,-1)
 \end{pspicture}
}

\newcommand{\exrsh}{%
 \psset{yunit=0.4cm}
 \psset{xunit=0.5cm}
 \begin{pspicture}(-3,-11)(3,1)%
   \psgrid[gridwidth=0.3pt,griddots=5,subgriddiv=1,gridlabels=5pt](-3,-11)(3,1)
  \psset{linewidth=1pt}%
  \psline{<-]}(-3,0)(1,0)
  \psline{[->}(1,0)(3,0)
  \psset{linewidth=1pt,linecolor=gray}%
  \psline{-}(3,-2)(1,-2)
  \psline{-}(1,-2)(-2,-11)
 \end{pspicture}
}

\newcommand{\exrsp}{%
 \psset{unit=0.5cm}
 \begin{pspicture}(-2,-1)(4,6)%
   \psgrid[gridwidth=0.3pt,griddots=5,subgriddiv=1,gridlabels=5pt](-1,-1)(4,6)
  \psline{[-]}(0,0)(3,0)
  \psset{linewidth=1pt,linecolor=gray}%
  \psline{-}(0,2)(3,5)
 \end{pspicture}
}

\newcommand{\recordi}{%
 \psset{yunit=0.4cm}
 \psset{xunit=0.5cm}
 \begin{pspicture}(-1,-5)(6,6)%
   \psgrid[gridwidth=0.3pt,griddots=5,subgriddiv=1,gridlabels=5pt](-1,-5)(6,6)
  \psline{[-]}(0,0)(5,0)
  \psset{linewidth=1pt,linecolor=gray}%
  \psline{-}(0,0)(5,6)
 \end{pspicture}
}

\newcommand{\recordii}{%
 \psset{yunit=0.4cm}
 \psset{xunit=0.5cm}
 \begin{pspicture}(-1,-5)(6,6)%
   \psgrid[gridwidth=0.3pt,griddots=5,subgriddiv=1,gridlabels=5pt](-1,-5)(6,6)
  \psline{[-]}(0,0)(5,0)
  \psset{linewidth=1pt,linecolor=gray}%
  \psline{-}(0,0)(1,2)(2,3)(3,3)(5,-5)
 \end{pspicture}
}

\def\defineTColor#1#2{%
  \newpsstyle{#1}{%
    fillstyle=vlines,hatchcolor=#2,%
    hatchwidth=0.1\pslinewidth,hatchsep=1\pslinewidth%
}}

\def\defineTHColor#1#2{%
  \newpsstyle{#1}{%
    fillstyle=hlines,hatchcolor=#2,%
    hatchwidth=0.1\pslinewidth,hatchsep=1\pslinewidth%
}}

\newgray{dgray}{0.65}
\defineTColor{TDGray}{dgray}
\defineTHColor{TD2Gray}{dgray}

\newcommand{\canDivNullh}{
 \psset{unit=0.5cm,Alpha=55,Beta=20}%
 \begin{pspicture}(-4,-6)(4,2)%
   \psset{linewidth=0.3pt}%
   \psset{style=TDGray}
   \pstThreeDLine(-3,-3,-3)(0,0,0)(0,-3,-3)
   \pstThreeDLine(-3,1,-3)(0,1,0)(0,0,0)(-3,-3,-3)
   \pstThreeDLine(0,4,-3)(0,1,0)(-3,1,-3)

   \psset{style=TD2Gray}
   \pstThreeDLine(0,-3,-3)(0,0,0)(3,0,-3)
   \pstThreeDLine(3,4,-3)(0,1,0)(0,4,-3)
   \pstThreeDLine(3,4,-3)(0,1,0)(0,0,0)(3,0,-3)
    \pstThreeDPlaneGrid[linewidth=0.1pt,subticks=3](0,0)(3,3)
    \pstThreeDPlaneGrid[linewidth=0.1pt,subticks=3](0,0)(-3,-3)
    \pstThreeDPlaneGrid[linewidth=0.1pt,subticks=3](0,0)(-3,3)
    \pstThreeDPlaneGrid[linewidth=0.1pt,subticks=3](0,0)(3,-3)
    \psset{linewidth=1.5pt,fillstyle=none}%
   \pstThreeDLine{-}(0,3,0)(0,1,0)(2,3,0)%
   \pstThreeDLine{-}(2,3,0)(0,1,0)(0,0,0)(3,0,0)%
   \pstThreeDLine{-}(3,0,0)(0,0,0)(0,-3,0)%
   \pstThreeDLine{-}(0,-3,0)(0,0,0)(-3,-3,0)%
   \pstThreeDLine{-}(-3, -3,0)(0,0,0)(0,1,0)(-3,1,0)%
   \pstThreeDLine{-}(-3, 1,0)(0,1,0)(0,3,0)%
 \end{pspicture}   
}

\newcommand{\canDivInftyh}{
 \psset{unit=0.5cm,Alpha=55,Beta=20}%
 \begin{pspicture}(-4,-6)(4,2)%
   \psset{linewidth=0.3pt}%
   \psset{style=TDGray}
   \pstThreeDLine{-}(-1,-4,-5)(-1,-1,-2)(-4,-4,-5)%
   \pstThreeDLine{-}(-4,-4,-5)(-1,-1,-2)(-4,-1,-5)%
   \psset{style=TD2Gray}
   \pstThreeDLine{-}(-4,-1,-5)(-1,-1,-2)(0,0,-2)(0,3,-5)%
   \pstThreeDLine{-}(3,0,-5)(0,0,-2)(-1,-1,-2)(-1,-4,-5)%
   \pstThreeDLine{-}(0,3,-5)(0,0,-2)(3,3,-5)%
   \pstThreeDLine{-}(3,3,-5)(0,0,-2)(3,0,-5)%
   \pstThreeDPlaneGrid[linewidth=0.1pt,subticks=3](0,0)(3,3)
   \pstThreeDPlaneGrid[linewidth=0.1pt,subticks=3](0,0)(-3,-3)
   \pstThreeDPlaneGrid[linewidth=0.1pt,subticks=3](0,0)(-3,3)
   \pstThreeDPlaneGrid[linewidth=0.1pt,subticks=3](0,0)(3,-3)
   \psset{linewidth=1.5pt,fillstyle=none}%
   \pstThreeDLine{-}(0,3,0)(0,0,0)(3,3,0)%
   \pstThreeDLine{-}(3,3,0)(0,0,0)(3,0,0)%
   \pstThreeDLine{-}(3,0,0)(0,0,0)(-1,-1,0)(-1,-3,0)%
   \pstThreeDLine{-}(-1,-3,0)(-1,-1,0)(-2,-2,0)%
   \pstThreeDLine{-}(-3,-3,0)(-1,-1,0)(-3,-1,0)%
   \pstThreeDLine{-}(-3,-1,0)(-1,-1,0)(0,0,0)(0,3,0)%
 \end{pspicture}   
}

\newcommand{\canDivOneh}{
 \psset{unit=0.5cm,Alpha=55,Beta=20}%
 \begin{pspicture}(-4,-6)(4,2)%
   \psset{linewidth=0.3pt}%
   \psset{style=TDGray}
   \pstThreeDLine{-}(1,-3,-3)(1,0,0)(0,0,0)(-3,-3,-3)%
   \pstThreeDLine{-}(-3,-3,-3)(0,0,0)(-3,0,-3)%
   \pstThreeDLine{-}(-3,0,-3)(0,0,0)(0,3,-3)%

   \psset{style=TD2Gray}
   \pstThreeDLine{-}(0,3,-3)(0,0,0)(1,0,0)(4,3,-3)%
   \pstThreeDLine{-}(4,3,-3)(1,0,0)(4,0,-3)%
   \pstThreeDLine{-}(4,0,-3)(1,0,0)(1,-3,-3)%
    \pstThreeDPlaneGrid[linewidth=0.1pt,subticks=3](0,0)(3,3)
   \pstThreeDPlaneGrid[linewidth=0.1pt,subticks=3](0,0)(-3,-3)
   \pstThreeDPlaneGrid[linewidth=0.1pt,subticks=3](0,0)(-3,3)
   \pstThreeDPlaneGrid[linewidth=0.1pt,subticks=3](0,0)(3,-3)
   \psset{linewidth=1.5pt,fillstyle=none}%
   \pstThreeDLine{-}(0,3,0)(0,0,0)(1,0,0)(3,2,0)%
   \pstThreeDLine{-}(3,2,0)(1,0,0)(3,0,0)%
   \pstThreeDLine{-}(3,0,0)(1,0,0)(1,-3,0)%
   \pstThreeDLine{-}(1,-3,0)(1,0,0)(0,0,0)(-3,-3,0)%
   \pstThreeDLine{-}(-3,-3,0)(0,0,0)(-3,0,0)%
   \pstThreeDLine{-}(-3,0,0)(0,0,0)(0,3,0)%
 \end{pspicture}   
}

\newcommand{\canDivNullp}{
 \psset{unit=0.55cm,Alpha=60,Beta=40}%
 \begin{pspicture}(-3,-3)(3,3)%
   \psset{linewidth=0.3pt}%
   \psset{style=TDGray}
   \pstThreeDPlaneGrid[linewidth=0.1pt,subticks=2](0,0)(2,2)
   \pstThreeDPlaneGrid[linewidth=0.1pt,subticks=2](0,0)(-2,-2)
   \pstThreeDPlaneGrid[linewidth=0.1pt,subticks=2](0,0)(-2,2)
   \pstThreeDPlaneGrid[linewidth=0.1pt,subticks=2](0,0)(2,-2)
   \pstThreeDLine{-}(-1,0,0)(0,-1,-1)(1,-1,-1)(1,0,0)(-1,0,0)
   \pstThreeDLine{-}(1,0,0)(0,1,0)(-1,1,0)(-1,0,0)(1,0,0)
   \psset{linewidth=1pt,fillstyle=none}%
   \pstThreeDLine{-}(1,0,0)(0,1,0)(-1,1,0)(-1,0,0)(0,-1,0)(1,-1,0)(1,0,0)
 \end{pspicture}   
}

\newcommand{\canDivInftyp}{
 \psset{unit=0.55cm}%
 \begin{pspicture}(-3,-3)(3,3)%
   \psset{linewidth=0.3pt}%
   \psset{style=TDGray}
   \pstThreeDPlaneGrid[linewidth=0.1pt,subticks=2](0,0)(2,2)
   \pstThreeDPlaneGrid[linewidth=0.1pt,subticks=2](0,0)(-2,-2)
   \pstThreeDPlaneGrid[linewidth=0.1pt,subticks=2](0,0)(-2,2)
   \pstThreeDPlaneGrid[linewidth=0.1pt,subticks=2](0,0)(2,-2)
   \pstThreeDLine{-}(1,-1,2)(1,0,1)(0,1,1)(-1,1,2)(1,-1,2)
   \pstThreeDLine{-}(-1,1,2)(-1,0,2)(0,-1,2)(1,-1,2)(-1,1,2)
   \psset{linewidth=1pt,fillstyle=none}%
   \pstThreeDLine{-}(1,0,0)(0,1,0)(-1,1,0)(-1,0,0)(0,-1,0)(1,-1,0)(1,0,0)
 \end{pspicture}   
}

\newcommand{\canDivOnep}{
 \psset{unit=0.55cm,Alpha=30,Beta=40}%
 \begin{pspicture}(-3,-3)(3,3)%
   \psset{linewidth=0.3pt}%
   \psset{style=TDGray}
   \pstThreeDPlaneGrid[linewidth=0.1pt,subticks=2](0,0)(2,2)
   \pstThreeDPlaneGrid[linewidth=0.1pt,subticks=2](0,0)(-2,-2)
   \pstThreeDPlaneGrid[linewidth=0.1pt,subticks=2](0,0)(-2,2)
   \pstThreeDPlaneGrid[linewidth=0.1pt,subticks=2](0,0)(2,-2)
   \pstThreeDLine{-}(0,-1,0)(1,-1,0)(1,0,0)(0,1,0)(0,-1,0)
   \pstThreeDLine{-}(0,-1,0)(-1,0,-1)(-1,1,-1)(0,1,0)(0,-1,0)
   \psset{linewidth=1pt,fillstyle=none}%
   \pstThreeDLine{-}(1,0,0)(0,1,0)(-1,1,0)(-1,0,0)(0,-1,0)(1,-1,0)(1,0,0)
 \end{pspicture}   
}

\newcommand{\cotangfannull}{%
 \psset{unit=0.5cm}
 \begin{pspicture}(-3.2,-3.2)(3.2,3.2)%
   \psgrid[gridwidth=0.3pt,griddots=5,subgriddiv=1,gridlabels=5pt](-3,-3)(3,3)

 \psset{linewidth=1pt}%

 \psline{-}(0,3)(0,1)(2,3)%
 \psline{-}(2,3)(0,1)(0,0)(3,0)%
 \psline{-}(3,0)(0,0)(0,-3)%
 \psline{-}(0,-3)(0,0)(-2,-2)%
 \psline{-}(-3, -3)(0,0)(0,1)(-3,1)%
 \psline{-}(-3, 1)(0,1)(0,3)%
\end{pspicture}}

\newcommand{\cotangfannullsmooth}{%
 \psset{unit=0.3cm}
 \begin{pspicture}(-3.2,-3.1)(3.2,3.1)%
   \psgrid[gridwidth=0.3pt,griddots=5,subgriddiv=1,gridlabels=5pt](-3,-3)(3,3)

 \psset{linewidth=1pt}%
 
 \psline{-}(0,3)(0,1)(2,3)%
 \psline{-}(2,3)(0,1)(0,0)(3,0)%
 \psline{-}(3,0)(0,0)(0,-3)%
 \psline{-}(0,-3)(0,0)(-2,-2)%
 \psline{-}(-3, -3)(0,0)(0,1)(-3,1)%
 \psline{-}(-3, 1)(0,1)(0,3)%

 \psline[linestyle=dashed]{-}(-3, 0)(0,0)
 \psline[linestyle=dashed]{-}(3, 1)(0,1)
\end{pspicture}}

\newcommand{\cotangfaninfty}{%
 \psset{unit=0.5cm}
 \begin{pspicture}(-3.2,-3.2)(3.2,3.2)%
   \psgrid[gridwidth=0.3pt,griddots=5,subgriddiv=1,gridlabels=5pt](-3,-3)(3,3)
  \psset{linewidth=1pt}%
  \psline{-}(0,3)(0,0)(3,3)%
  \psline{-}(2,2)(0,0)(0,0)(3,0)%
  \psline{-}(3,0)(0,0)(-1,-1)(-1,-3)%
  \psline{-}(-1,-3)(-1,-1)(-2,-2)%
  \psline{-}(-3, -3)(-1,-1)(-3,-1)%
  \psline{-}(-3, -1)(-1,-1)(0,0)(0,3)%
\end{pspicture}}

\newcommand{\cotangfaninftysmooth}{%
 \psset{unit=0.3cm}
 \begin{pspicture}(-3.2,-3.1)(3.2,3.1)%
   \psgrid[gridwidth=0.3pt,griddots=5,subgriddiv=1,gridlabels=5pt](-3,-3)(3,3)
  \psset{linewidth=1pt}%
  \psline{-}(0,3)(0,0)(3,3)%
  \psline{-}(2,2)(0,0)(0,0)(3,0)%
  \psline{-}(3,0)(0,0)(-1,-1)(-1,-3)%
  \psline{-}(-1,-3)(-1,-1)(-2,-2)%
  \psline{-}(-3, -3)(-1,-1)(-3,-1)%
  \psline{-}(-3, -1)(-1,-1)(0,0)(0,3)%
  \psline[linestyle=dashed]{-}(0, -3)(0,0)
  \psline[linestyle=dashed]{-}(-1, 3)(-1,-1)
\end{pspicture}}

\newcommand{\cotangfanone}{%
 \psset{unit=0.5cm}
 \begin{pspicture}(-3.2,-3.2)(3.2,3.2)%
\psgrid[gridwidth=0.3pt,griddots=5,subgriddiv=1,gridlabels=5pt](-3,-3)(3,3)
 \psset{linewidth=1pt}%
 \psline{-}(0,3)(0,0)(1,0)(3,2)%
 \psline{-}(3,2)(1,0)(3,0)%
 \psline{-}(3,0)(1,0)(1,-3)%
 \psline{-}(1,-3)(1,0)(0,0)(-3,-3)%
 \psline{-}(-2, -2)(0,0)(-3,0)%
 \psline{-}(-3, 0)(0,0)(0,3)%
\end{pspicture}}

\newcommand{\cotangfanonesmooth}{%
 \psset{unit=0.3cm}
 \begin{pspicture}(-3.2,-3.1)(3.2,3.1)%
\psgrid[gridwidth=0.3pt,griddots=5,subgriddiv=1,gridlabels=5pt](-3,-3)(3,3)
 \psset{linewidth=1pt}%
 \psline{-}(0,3)(0,0)(1,0)(3,2)%
 \psline{-}(3,2)(1,0)(3,0)%
 \psline{-}(3,0)(1,0)(1,-3)%
 \psline{-}(1,-3)(1,0)(0,0)(-3,-3)%
 \psline{-}(-2, -2)(0,0)(-3,0)%
 \psline{-}(-3, 0)(0,0)(0,3)%
 \psline[linestyle=dashed]{-}(1, 0)(1,3)
 \psline[linestyle=dashed]{-}(0, -3)(0,0)
\end{pspicture}}

\newcommand{\tropical}[1]{{\mathcode`\*="020C\mathcode`\+="0208\renewcommand{\cdot}{\odot}\renewcommand{\sum}{\bigoplus}
$#1$}}
\newcommand{\mtrop}[1]{\text{\tropical{#1}}}

\begin{document}

\maketitle

\begin{abstract}
A description of complete normal varieties with lower dimensional torus action has been given in \cite{divfans}, generalizing the theory of toric varieties. Considering the case where the acting torus $T$ has codimension one, we describe $T$-invariant Weil and Cartier divisors and provide formulae for calculating global sections, intersection numbers, and Euler characteristics. As an application, we use divisors on these so-called $T$-varieties to define new evaluation codes called $T$-codes. We find estimates on their minimum distance using intersection theory. This generalizes the theory of toric codes and combines it with AG codes on curves. 
As the simplest application of our general techniques we look at codes
on ruled surfaces coming from decomposable vector bundles. Already this
construction gives codes that are better than the related product code.
Further examples show that we can improve these codes by constructing
more sophisticated $T$-varieties.  These results suggest to look further
for good codes on $T$-varieties.\end{abstract}

\section{Introduction}\label{sec:introduction}
An important class of linear codes is the class of Algebraic Geometry Codes, introduced by Goppa in 1981. These codes arise by evaluating global sections of a line bundle on a curve over  $\mathbb{F}_q$ at a number of $\mathbb{F}_q$-rational points; good estimates on the dimension and minimum distance of such codes can be obtained by using the theorem of Riemann-Roch. Such codes have since been generalized to higher-dimensional varieties. It is however often difficult to obtain non-trivial estimates on the parameters of such codes. One class of varieties where non-trivial estimates have been made is that of toric varieties, which one can describe combinatorially.

Toric varieties have been generalized in \cite{MR2207875} and \cite{divfans} to so-called $T$-varieties, which are normal varieties admitting an effective $m$-dimensional torus action. $T$-varieties can then be described by a variety $Y$ of dimension $\dim X -m$ along with combinatorial data called a divisorial fan. If the acting torus has codimension one, $Y$ is then a curve. The aim of this paper is to analyze certain evaluation codes on such varieties; we shall call these codes $T$-codes. 

In short, a $T$-code over $\mathbb{F}_q$ is constructed from:
\begin{itemize}
\item
a curve $Y$ over $\mathbb{F}_q$;
\item
a so-called \emph{divisorial polytope} (cf. definition \ref{def:divpoly}), essentially a concave function $h^*:\Box_h\to \wdiv_\QQ Y$ where $\Box_h$ is a polytope with vertices in some lattice $M\cong \mathbb{Z}^m$ and $h^*$ satisfies some additional conditions;
\item and a set $\mathcal{P}=\{P_1,\ldots,P_l\}$ of $\mathbb{F}_q$-rational points on $Y$.
\end{itemize}
Assuming that the support of $h^*(u)$ is disjoint from $\mathcal{P}$ for each $u\in \Box_h\cap M$, we can define the $T$-code $\mathcal{C}(Y,h^*,\mathcal{P})$ as the sum of a number of product codes:
 $$\mathcal{C}(Y,h^*,\mathcal{P}):= \sum_{u \in \Box_h \cap M} \mathcal{C}_u \otimes \mathcal{C}(Y,h^*(u),\mathcal{P})$$
where $\mathcal{C}_u$ is the $[(q-1)^m,1,(q-1)^m]$ code generated by $\left(t^{u}\right)_{t \in {(\mathbb{F}_q^*)}^m}$ and $\mathcal{C}(Y,h^*(u),\mathcal{P})$ is the AG code corresponding to the curve $Y$, divisor $h^*(u)$, and point set $\mathcal{P}$. By interpreting $\mathcal{C}(Y,h^*,\mathcal{P})$ as the image under a linear map of the Riemman-Roch space of a divisor on a $T$-variety, we are able to give non-trivial estimates for the dimension $k$ and minimum distance $d$ of this code.

We begin in section \ref{sec:theory-t-varieties} by recalling the basic theory of $T$-varieties. We then proceed to describe divisors and intersection theory on $T$-varieties in section \ref{sec:divisors_and_intersections}. In particular, we describe all $T$-invariant Cartier and Weil divisors combinatorially, calculate the global sections of a $T$-invariant Cartier divisor, and determine exactly when a $T$-Cartier divisor is (semi-)ample. Furthermore, we provide formulae for calculating intersection numbers and for the Euler characteristic of a line bundle. The theory of this section is analogue to that of divisors on toric varieties and is essential for estimating the parameters of the evaluation codes we construct.

In section \ref{sec:t-codes}, we define $T$-codes and show how to estimate dimension and minimum distance, providing upper and lower bounds for both parameters. We give special attention to the case of two-dimensional $T$-varieties, where we provide a better lower bound for the minimum distance.

Finally, we provide a number of examples in section \ref{sec:examples}. We first consider $T$-codes coming from those ruled surfaces corresponding to a rank two decomposable vector bundle. In particular, we show that some of these codes have better parameters than those estimated for the product of a Reed-Solomon and a one-point Goppa code. In a second example, we show how one can use the Hasse-Weil bound to improve the lower bound on the minimum distance. This example also shows that there are better $T$-codes than those coming from ruled surfaces. In a final example, we describe a $T$-code over $\mathbb{F}_7$ whose parameters are as good as any known linear code.  

\section{The Theory of $T$-Varieties}
\label{sec:theory-t-varieties}
First we recall some facts and notations from convex geometry.
Here, $N$ always is a lattice and $M:=\Hom(N,\ZZ)$ its dual.
The associated $\QQ$-vector spaces $N \otimes \QQ$ and $M \otimes \QQ$ are denoted by $N_\QQ$ and $M_\QQ$ respectively.
Let $\sigma \subset N_\QQ$ be a pointed convex polyhedral cone.
A polyhedron $\Delta$ which can be written as a Minkowski sum 
$\Delta = \pi + \sigma$ of $\sigma$ and a compact polyhedron $\pi$ is said to have $\sigma$ as its tail cone.

With respect to Minkowski addition the polyhedra with tail cone $\sigma$ form a semigroup which we denote by $\Pol_{\sigma}^+(N)$.
Note that $\sigma \in \Pol_{\sigma}^+(N)$ is the neutral element of this semigroup and that $\emptyset$ is by definition also an element of $\Pol_{\sigma}^+(N)$.

A polyhedral divisor with tail cone $\sigma$ on a normal variety $Y$ is a 
formal finite sum
$$\D = \sum_D \Delta_D \otimes D,$$
where $D$ runs over all prime divisors on $Y$ and $\Delta_D \in \Pol^+_\sigma$.
Here, finite means that only finitely many coefficients differ from the 
tail cone.

We may evaluate a polyhedral divisor for every element $u \in \sigma^\vee \cap M$ via
$$\D(u):=\sum_D \min_{v \in \Delta_D} \langle u , v \rangle D$$
in order to obtain an ordinary divisor on $\loc \D$. Here, 
$ \loc \D:= Y \setminus \left( \bigcup_{\Delta_D = \emptyset} D \right)$ denotes the locus of $\D$.

\begin{defn}
  A polyhedral divisor $\D$ is called {\em Cartier} if 
  every evaluation $\D(u)$, $u \in \sigma^\vee \cap M$, is Cartier.
\end{defn}

To a Cartier polyhedral divisor we associate a $M$-graded $k$-algebra sheaf and consequently an affine scheme over $\loc \D$ admitting a $T^M$-action:
$$\tilde{X}:=\tilde{\pdv}(\D):= \Spec_{\loc \D} \bigoplus_{u \in \sigma^\vee \cap M} \CO(\D(u)).$$

From \cite{MR2207875} we know that this construction gives a normal variety of dimension $\dim N + \dim Y$ admitting a torus action of $T^N$ with $\loc \D$ as its good quotient.
 
Moreover, for every affine normal variety $X$ there exists a polyhedral divisor $\D$ such that $X = \spec \Gamma(\tilde{\pdv}(\D), \CO_{\tilde{\pdv}(\D)})$. $X$ and $\tilde{X}$  coincide if $X$ admits a torus action with a good quotient.

\begin{defn}
Let $\D=\sum_D \Delta_D \otimes D$, $\D'=\sum_D \Delta'_D \otimes D$ be two polyhedral divisors on $Y$.
\begin{enumerate}
\item We write $\D' \subset \D$ if $\Delta'_D \subset \Delta_D$ holds for every prime divisor $D$.
\item We define the intersection of polyhedral divisors $$\D \cap \D' := \sum_D (\Delta'_D \cap \Delta_D) \otimes D.$$
\item We define the degree of a polyhedral divisor $$\deg \D := \sum_D \Delta_D.$$
\item For a (not necessarily closed) point $y \in Y$ we define the fibre polyhedron $\Delta_y := \D_y := \sum_{y \in D} \Delta_D$.
\item We call $\D'$  a {\em face} of $\D$ and write $\D' \prec \D$ if $\D'_y$ is a face of $\D_y$ for every $y \in Y$.
\end{enumerate}
\end{defn}

Assume $\D' \subset \D$. This implies
$$\bigoplus_{u \in \sigma^\vee \cap M} \CO(\D'(u))
\hookleftarrow \bigoplus_{u \in \sigma^\vee \cap M}  \CO(\D(u)))$$ and we get a dominant morphism $\tilde{\pdv}(\D') \rightarrow \tilde{\pdv}(\D)$.

\begin{prop}[\cite{divfans}, Prop. 3.4, Rem. 3.5]
  This morphism defines an open embedding if and only if $\D' \prec \D$ holds.
\end{prop}

\begin{defn}\label{sec:def-fansy-div}
  Consider a smooth projective curve $Y$. A {\em fansy divisor} is 
  a formal finite sum
  $$\fan = \sum_{P \in Y} \fan_P \otimes Z$$
such that:
  \begin{enumerate}
  \item $\fan_P$ are polyhedral subdivisions covering $N_\QQ$ and sharing a common tail fan;
  \item Finite means here that for all but finitely many points, $\fan_P$ equals the tail fan.
  \end{enumerate}
\end{defn}

Consider a finite set of polyhedral divisors $\mathcal{S}$, such that $\D \succ \D' \cap \D \prec \D'$ for every pair $\D,\D' \in \mathcal{S}$.
Assume furthermore that their polyhedral coefficients $\D_P$ form the subdivisions $\fan_P$ of a fansy divisor.

From such a set we may construct a scheme $\tilde{\pdv}(\fan)$ by gluing $\pdv(\D)$s via $$\tilde{\pdv}(\D) \leftarrow \tilde{\pdv}(\D \cap \D') \rightarrow \tilde{\pdv}(\D').$$
Note that we had to check the cocycle condition, this is done in 
\cite[Thm. 5.3]{divfans}. From theorem~7.5 ibid. we know that we get a complete variety this way.

This variety is uniquely determined by the underlying fansy divisor. Different sets $\mathcal{S}$ correspond to different open coverings. Therefore, we may denote the resulting variety by $\tilde{\pdv}(\fan)$. 

Theorem~5.6 in \cite{divfans} tell us that for every normal $T$-variety $X$  with $\dim X = \dim T +1$ we may find a fansy divisor $\fan$ and a proper birational map $\tilde{\pdv}(\fan) \rightarrow X$.
If $X$ admits a good quotient under the torus action this morphism turns out to be the identity.

\begin{rem}
\label{sec:rem-suitable-covering}
For a fansy divisor $\fan$ and an open covering $\{U_i\}_{i\in I}$ of $Y$ we can find a set $\mathcal{S}$ as above, such that for every $\D \in \mathcal{S}$ there is a $i \in I$ such that $\loc \D = U_i$.
\end{rem}

\begin{exmp}
  Let $Y$ be a smooth projective curve and $Q_1,Q_2 \in Y$ two points. We consider the fansy divisor $\fan$ given by the coefficients in figure \ref{fig:surface-fd}. $\tilde{\pdv}(\fan)$ is a complete surface with one dimensional torus action.
  \begin{figure}[htbp]
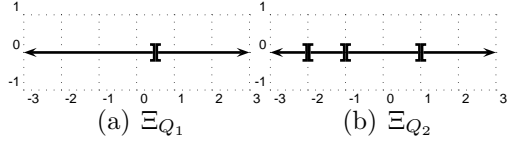

    \centering
    \subfigure[$\fan_{Q_1}$]{\exmpi}
    \subfigure[$\fan_{Q_2}$]{\exmpii}
    \caption{The fansy divisor of a surface}\label{fig:surface-fd}
  \end{figure}
\end{exmp}

\begin{exmp}
  We consider the fansy divisor on $\mathbb{P}^1$ given by the coefficients in figure~\ref{fig:threefold}. $\tilde{\pdv}(\fan)$ is a complete (singular) threefold with two-dimensional torus action.
  \begin{figure}[htbp]
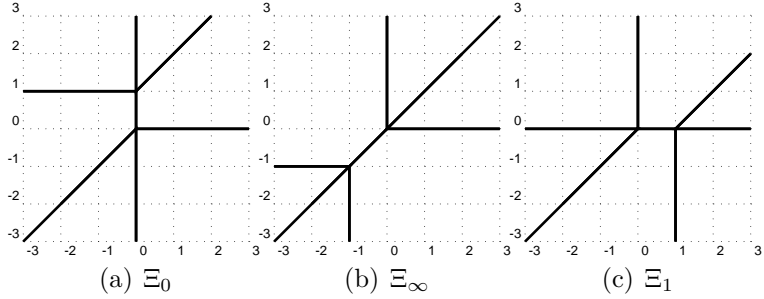

    \centering
    \subfigure[$\fan_0$]{\cotangfannull}
    \subfigure[$\fan_\infty$]{\cotangfaninfty}
    \subfigure[$\fan_1$]{\cotangfanone}
    \caption{The fansy divisor of a threefold}\label{fig:threefold}
  \end{figure}
\end{exmp}

\section{Divisors and Intersection Theory on $T$-Varieties}
\label{sec:divisors_and_intersections}
From now on we shall only consider torus actions of codimension one; we will study them via fansy divisors.

\subsection{Cartier divisors}
\label{sec:cartier-divisor}
Let $\Sigma \subset N_\QQ$ be a complete polyhedral subdivision of $N$ consisting of tailed polyhedra. We consider continuous functions
$h:|\Sigma| \rightarrow \QQ$  which are affine on every polyhedron in $\Sigma$. Let $\Delta \in \Sigma$ be a polyhedron with tail 
cone $\delta$. Then $h$ induces a linear function $h^\Delta_0$ on $\delta = \tail \Delta$ by defining $h^\Delta_0(v):= h(P+v)-h(P)$ for some $P \in \Delta$. We call $h^\Delta_0$ the linear part of $h|_\Delta$.

\begin{defn}
  An (integral) \textit{support function} on a polyhedral subdivision $\Sigma$ is a piecewise affine function as above with integer slope and integer translation. To be precise: for $v \in |\Sigma|$ and $k \in \NN$ such that $kv$ is a lattice point we have $k h(v) \in \ZZ$.
The group of support functions on $\Sigma$ is denoted by $\SF_\Sigma$.
\end{defn}

Let $\fan$ be a divisorial fan on a curve $Y$. For every $P \in Y$ we get a polyhedral subdivision $\fan_P$ consisting of polyhedral coefficients.  We consider $\SF(\fan)$, the group of  formal sums $\sum_{P \in Y} h_P P$ with 
\begin{enumerate}
\item $h_P \in \SF_{\fan_P}$ a support function of the $P$-slice of $\fan$.
\item all $h_P$ have the same linear part $h_0$.
\item $h_P$ differs from $h_0$ for only finitely many points $P \in Y$.\\
{\em We refer to this fact by calling this sum finite and we omit those summands which equal $h_0$.}
\end{enumerate}

\begin{defn}
	A support function $h \in \SF(\fan)$ is called principal if $h(v) = \langle u, v \rangle + D$, with $u \in M$ and $D$ is a principal divisor on $Y$.
By  $h(v) = \langle u, v \rangle + D$ we mean that $h_P(v)= \langle u, v \rangle + a_P$, where $D=\sum_P a_P P$.
\end{defn}

If $h=\sum h_P P \in \SF(\fan)$ we consider a covering
$\{Y_i\}$ of $Y$ such that $P$ is a principal divisor on the $Y_i$ for every $P \in Y$ with $h_P \neq h_0$, and such that every $Y_i$ contains at most one of these points.

 We may find a set $\mathcal{S}$ as above which is compatible with this covering and induces $\fan$.
\label{sec:construction-cartier-div}
Now we choose a $\D \in \mathcal{S}$ with $\loc \D = Y_i$ and $h_P \neq h_0$. $h_P$ is an affine function on every polyhedron in $\fan_P$ so we get $-h_P|_{\D_P}(v) = \langle v, u \rangle + a$ for some $u \in M$ and $a\in \ZZ$. Assume that $\divisor(f) = a P$ on $Y_i$; then $f \cdot \chi^u \in K(\tilde{\pdv}(\D))^T$ 
defines a $T$-invariant principal divisor $H_\D$ on $\tilde{\pdv}(\D)$. These principal divisors fit together to a Cartier divisor $D_h$ on $\tilde{\pdv}(\fan)$.
Here $K(\tilde{\pdv}(\D))^T := \bigoplus_{u \in M} K(Y) \cdot \chi^u \supset \Gamma(\tilde{\pdv}(\D))$ denotes the ring of invariant rational functions on $\tilde{\pdv}(\D)$. 
In this way the group of integral support functions on $\fan$ corresponds to that of invariant Cartier divisors on $\tilde{\pdv}(\fan)$.

\subsection{Weil divisors}
\label{sec:weil-divisors}
In general there are two types of $T$-invariant prime divisors, namely those which consist
\begin{enumerate}
\item of orbit closures of dimension $\dim T$;\label{item:dimT}
\item and of orbit closures of dimension $\dim T -1$.\label{item:dimT-1}
\end{enumerate}

\begin{prop}
If $\D$ is a polyhedral divisor on a curve with tailcone $\sigma$, there are one-to-one correspondences 
\begin{enumerate}
\item between prime divisors of type~\ref{item:dimT} and pairs $(P,v)$ with $P$ a point on $Y$ 
  and $v$ a vertex of $\Delta_P$;
\item between prime divisors of type~\ref{item:dimT-1} and rays $\rho$ of $\sigma$ with $\deg \D \cap \rho = \emptyset$.
\end{enumerate}
\end{prop}

\begin{proof}
	Consider the quotient map $\pi:\tilde{\pdv} \rightarrow \loc \D$. In \cite{MR2207875} the orbit structure of the fibres of $\pi$ is described. Thus, we know that faces $F \prec \D_y$ correspond to $T$-invariant subvarieties  of codimension $\dim(F)$ in $\pi_y:=\pi^{-1}(y)$. The correspondences follow by using this for closed points and the generic point, respectively.
\end{proof}

 \begin{rem}
   We may also describe the ideals of prime divisors in terms of polyhedral divisors:
   \begin{enumerate}
   \item For prime divisors of type~\ref{item:dimT} corresponding to a vertex $(P,v)$, the ideal is given by
     $$I_{P,v} = \bigoplus_{u \in \sigma^\vee} \Gamma(Y,\CO(\D(u))) \cap \{f \mid \ord_P(f) > \langle v , u \rangle\}.$$
   \item For prime divisors of type~\ref{item:dimT-1}, the corresponding ideal is generated by all multidegrees which are
     not orthogonal to $\rho$:
     $$I_{\rho} = \bigoplus_{u \in \sigma^\vee  \setminus \rho^\perp} \Gamma(Y,\CO(\D(u))).$$
  \end{enumerate}
 \end{rem}

\begin{prop}\label{prop:cartier2weil}
	Let $h = \sum_P h_P$ correspond to the Cartier divisor $D_h$ on $\tilde{X}(\D)$. The corresponding Weil divisor is given by 
  $$-\sum_{\rho} h_0 (n_\rho) \rho - \sum_{(P,v)} \mu(v) h_P(v) (P,v),$$
 where $\mu(v)$ is the smallest integer $k \geq 1$ such that $k \cdot v$ is a lattice point. This lattice point is a multiple of the primitive lattice vector $n_v$: $\mu(v)v=\varepsilon(v) n_v$.
\end{prop}

\begin{proof}
  This is a local statement, so we will pass to a sufficiently small invariant open affine set which meets a particular prime divisor.
  If we translate this to our combinatorial language
and we consider a prime divisor corresponding to $(P,v)$ or $\rho$ then we have to choose a polyhedral divisor $\D' \prec \D \in\mathcal{S}$ such that $v$ is also a vertex of $\D'_P$ or $\rho$ is a ray in $\tail \D'$, respectively.

  So we restrict to following two (affine) cases:
  \begin{enumerate}
  \item  $\D$ is a polyhedral divisor with tail cone $\sigma=0$ and a single point $\Delta_P =\{v\} \subset N$ as the only nontrivial coefficient. Moreover, $Y$ is affine and factorial. In particular, $P$ is a prime divisor with (local) parameter $t_P$.
  \item $\D$ is the trivial polyhedral divisor with one dimensional tail cone $\rho$ over an affine locus $Y$.
  \end{enumerate}  
 
  In the first case we may choose $\ZZ$-Basis $e_1,\ldots, e_m$ of $N$ with 
  $e_1 = n_v$. Consider the dual basis $e^*_1,\ldots, e^*_m$.
By definition $\varepsilon(v)$ and $\mu(v)$ are coprime so we will find $a,b \in \ZZ$ such that
  $a \mu(v) + b \varepsilon(v) = 1$. In this situation $y:=t_P^a\chi^{be^*_1}$ is irreducible in 
  $$\Gamma(\CO_X) = \Gamma(\CO_Y)[y,t_P^{\pm \varepsilon(v)}\chi^{\mp \mu (v) e^*_1},\chi^{\pm e^*_2}, \ldots ,\chi^{\pm e^*_m}]$$ 
   and defines the prime divisor $(P,v)$.
  We consider an element $t_P^\alpha \chi^{u}$ with $u=\sum_i \lambda_i e^*_i$. 
  The $y$-order of $t_P^\alpha \chi^{u}$ is $$\varepsilon(v)\lambda_1 + \mu(v)\alpha = \mu(v)(\langle u, v \rangle+\alpha),$$ because
  $t_P^\alpha \chi^{u} = y^{\varepsilon(v)\lambda_1 + \mu(v)\alpha} (t_P^{- \varepsilon(v)}\chi^{\mu (v) e^*_1})^{\lambda_1 a + b \alpha}$, and
  $(t_P^{- \varepsilon(v)}\chi^{\mu (v) e^*_1})$ is a unit.
  
  In the second case we choose a $\ZZ$-basis  $e_1,\ldots, e_m$ of $N$ with $e_1 = n_\rho$. We once again consider the dual basis $e^*_1,\ldots, e^*_m$. In this situation
  $$\Gamma(\CO_X) = \Gamma(\CO_Y)[\chi^{e^*_1},\chi^{\pm e^*_2}, \ldots ,\chi^{\pm e^*_m}].$$
  Now $(\chi^{e^*_1})$  defines the prime divisor $\rho$ on $X$. For a principal divisor $f \cdot \chi^u$, the $\chi^{e^*_1}$-order equals the $e^*_1$-component of $u$, i.e. $\langle u, n_\rho \rangle$.
\end{proof}

\begin{exmp}
  For our threefold example we consider $D_h$ where $h_0,h_\infty,h_1$ are given by the tropical polynomials
\\
\begin{align*}
h_0=&\mtrop{0*x^{(-1,0)}+0*x^{(-1,1)}+0*x^{(0,1)}+0*x^{(1,0)}+1*x^{(1,-1)}+1*x^{(0,-1)}}\\
h_\infty=&\mtrop{(-2)*x^{(-1,0)}+(-2)*x^{(-1,1)}+(-1)*x^{(0,1)}+(-1)*x^{(1,0)}+}\\
&\qquad\qquad\qquad\qquad\qquad\qquad\qquad\mtrop{+(-2)*x^{(1,-1)}+(-2)*x^{(0,-1)}}\\
h_1=&\mtrop{1*x^{(-1,0)}+1*x^{(-1,1)}+0*x^{(0,1)}+0*x^{(1,0)}+0*x^{(1,-1)}+0*x^{(0,-1)}}
\end{align*}
where we are using the tropical semi-ring with operations $\oplus=\min,\odot=+$. These support functions are pictured in figure \ref{fig:3fold-support}. The Weil divisor corresponding to $D_h$ is $\sum_\rho D_\rho + 2D_{(\infty,0)}+ 2D_{(\infty,(-1,-1))}$. This is the anti-canonical divisor of $X:=\tilde{\pdv}(\fan)$ \cite{petersen-suess08}.

  \begin{figure}[htbp]
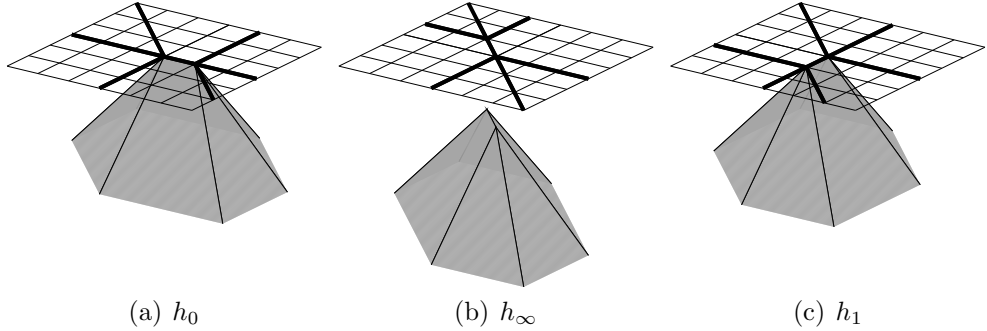

    \centering
    \subfigure[$h_0$]{\canDivNullh}
    \subfigure[$h_\infty$]{\canDivInftyh}
    \subfigure[$h_1$]{\canDivOneh}
\caption{Support functions for a $T$-threefold}\label{fig:3fold-support}
  \end{figure}
\end{exmp}

  \subsection{Global sections}
  For a support function $h$ on $X$ we may consider the $M$-graded vector space of global sections of $D_h$
  $$L(D_h) = \bigoplus_{u \in M} L(D_h)_u := \Gamma(X,\CO(D_h)).$$
  The {\em weight set} of $L(D_h)$ is defined as the set $\{u \in M \mid L(D)_u \neq 0\}$. 
  For a Cartier divisor given by $h \in \tcadiv(\fan)$ we will bound its weight set by a polyhedron as well as describe 
  the graded module structure of $L(D)$.

Consider a support function $h = \sum_P h_P P$ with linear part $h_0$. We define its associated polytope 
$$\Box_h := \Box_{h_0}:=\{u \in M_\mathbb{Q} \mid  \langle u, v \rangle  \geq  h_0(v) \; \forall_{v \in N}\}$$
 and associate a dual function $h^*:\Box_h \rightarrow \wdiv_\QQ Y$ via
$$h^*(u):=\sum_P h_P^*(u)P := \sum_P \minvert(u-h_P)P,$$ 
where $\minvert(u - h_P)$ denotes the minimal value of $u - h_P$ on the vertices of $\fan_P$.

\begin{rem} Let $h$ be a concave support function. Every affine piece of $h_P$ corresponds to a pair $(u,-a_u) \subset M \times \ZZ$. $h^*_P$ is defined to be the coarsest concave piecewise affine function with $h^*_P(u)=a_u$.

We can reformulate this in terms of the tropical semi-ring with operation $\oplus=\min,\odot=+$. We might think of the $h_P$ as given by tropical polynomials \tropical{\sum_{w \in I} (-a_w) * x^w}, then $\Box_{h}=\conv(I)$ and $h_P^*(w)=a_w$, i.e. $\graph_{h^*_P}$ is the reflected lower newton boundary of the tropical polynomial for $h_P$.
\end{rem}

\begin{defn}\label{def:divpoly}
  A {\em divisorial polytope} $h^*$
is a pair consisting of an ordinary  polytope $\Box_h \subset M_\QQ$ and a concave piecewise affine function $h^*:\Box_h \rightarrow \wdiv_\QQ Y$ such that
\begin{enumerate}
\item 
$\deg h^*(u) \geq 0$ for all vertices $u$ of $\Box_h$,
\item 
some multiple of $h^*(u)$ is principal in case of $\deg h^*(u) = 0$ for a vertex $u$.
\item $\Box_h$ is a lattice polytope as is $\conv (\graph_{h^*_P})$ for each $P\in Y$.\label{item:integral}
\end{enumerate}
\end{defn}

Let $\Box_g,\Box_h\in M_\mathbb{Q}$ be polytopes. For any concave piecewise affine functions $g^*:\Box_g\rightarrow \wdiv_\QQ Y$ and  $h^*:\Box_h \rightarrow \wdiv_\QQ Y$ we define their sum $g^*+h^*$ to be the piecewise affine concave function on $\Box_g + \Box_h$ given by
  $$(g^*_P + h^*_P)(u)=\max\{h^*_P(w)+g^*_P(w') \mid u=w+w'\}.$$

\begin{rem}\label{rem:h*sum}
For $g,h\in SF(\fan)$, one easily checks that $$\Box_g+\Box_h\subset \Box_{g+h}$$ and that $$g_P^*(u)+h_P^*(u)\leq (g+h)_P^*(u)$$ for all $P\in Y$ and all $u\in
 \Box_g+\Box_h$. Furthermore, if $h_P$ and $g_P$ are convex, they correspond to tropical polynomials \tropical{f},\tropical{f'}. It follows then that $(g+h)_P$ corresponds to \tropical{f \cdot f'}. Its reflected lower newton boundary is exactly the graph of $(g+h)_P^*$, thus the equality
  $$(g+h)_P^* = g_P^*+h_P^*$$
  holds.
\end{rem}

To a divisorial polytope $h^*$ we might associate a fansy divisor $\fan$ and support function $h$ on $\fan$ such that $h^*$ corresponds to $h$ in the way given above. Indeed, to every $h^*_P$ we can associate a tropical polynomial 
  \tropical{f:=\sum_{(u,a_u)} (-a_u) * x^u}, where $(u,a_u)$ runs over the vertices of $\graph_{(h^*_P)}$. This polynomial induces via evaluation a piecewise affine function and a polyhedral subdivision $\fan_P$ of $N$.
  
\begin{rem}
If we remove condition \ref{item:integral} from the definition of a divisorial polytope (definition \ref{def:divpoly}), the association in the above paragraph gives us a $\mathbb{Q}$-Cartier divisor.
\end{rem}

For every fansy divisor there exists a smooth refinement, i.e. a fansy divisor $\fan'$ such that every $\fan'_P$ is a refinement of $\fan_P$ and  $\tilde{\pdv}(\fan')$ is smooth \cite{suess08}. Every support function $h$ on $\fan$ is obviously also a support function on $\fan'$. Thus, for a given divisorial polytope $h^*$ we might alway consider a smooth fansy divisor $\fan$ and a support function $h$ on it such that the associated dual function equals $h^*$.

\begin{exmp}
  We now revisit our threefold example. Figure \ref{fig:3fold-h*} shows a sketch of $h^*$. We show a refinement of the fansy divisor in figure \ref{fig:3fold-refined} which gives a smooth threefold. 
    \begin{figure}[htbp]
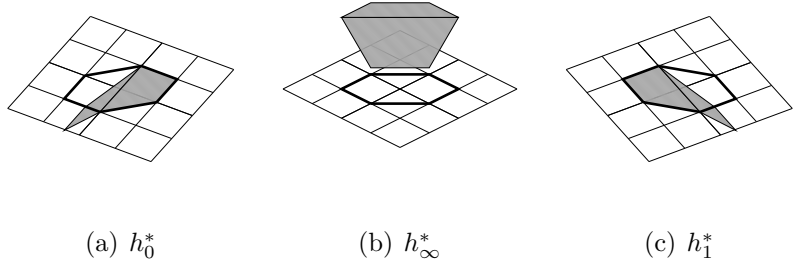

      \centering
      \subfigure[$h^*_0$]{\canDivNullp}
      \subfigure[$h^*_\infty$]{\canDivInftyp}
      \subfigure[$h^*_1$]{\canDivOnep}
	\caption{$h^*$ for a $T$-threefold}\label{fig:3fold-h*}
    \end{figure}
  
    \begin{figure}[htbp]
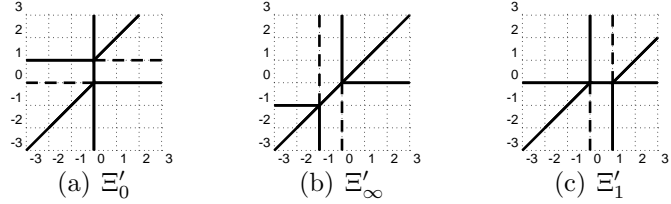

      \centering
      \subfigure[$\fan'_0$]{\cotangfannullsmooth}\hspace{3em}
      \subfigure[$\fan'_\infty$]{\cotangfaninftysmooth}\hspace{3em}
      \subfigure[$\fan'_1$]{\cotangfanonesmooth}
      \caption{A refined polyhedral divisor}\label{fig:3fold-refined}
\end{figure}

\end{exmp}

  \begin{prop}\label{prop:global_sections}
    Let $h \in SF(\fan)$  be a Cartier divisor with linear part $h_0$. Then
    \begin{enumerate}
    \item The weight set of $L(D_h)$ is a subset of $\Box_{h}$.
    \item for $u \in \Box_{h}$ we have $$L(D_h)_u = \Gamma(Y,\CO(h^*(u))).$$
    \end{enumerate}
  \end{prop}

  \begin{proof}
    By definition of $\CO(D_h)$ we have
    $$\Gamma(X,\CO(D_h))^{T} = \{ \chi^u f  \mid  \divisor(\chi^u f) - \sum_{\rho} h_0 (n_\rho) \rho - \sum_{(P,v)} \mu(v) h_P(v) (P,v) \geq 0\}.$$ But $\divisor(\chi^u f) = \sum_{\rho} \langle u, n_\rho \rangle \rho + \sum_{(P,v)} \mu(v) (\langle u, v \rangle + \ord_P(f)) (P,v)$, so for $\chi^u f \in L(h)$ we get the following bounds:
    \begin{enumerate}
    \item $\langle u, n_\rho \rangle \geq h_0(n_\rho) \; \forall_\rho$
    \item $\ord_P(f) + \langle u, v \rangle \geq  h_P(v) \; \forall_{(P,v)}$
    \end{enumerate}
    The first implies that $u \in \Box_{h}\cap M$, the second that $\ord_P(f) + (u - h_P)(v) \geq 0  \; \forall\  {(P,v)}$.
  \end{proof}

\begin{defn}
  For a cone $\sigma \in \fan_0^{(n)}$ of maximal dimension in the tail fan and a $P \in Y$ we get exactly one polyhedron $\Delta^\sigma_P \in \fan_P$ having tail $\sigma$.
  
  For a given concave support function $h = \sum h_P P$ 
  We have 
  $$h_P|\Delta^\sigma_P = \langle \cdot, u^h(\sigma) \rangle + a^h_P(\sigma).$$
  The constant part gives rise to a divisor on $Y$:
  $$h|_{\sigma}(0) := \sum_P a^h_P(\sigma) P.$$
\end{defn}

\begin{prop}
\label{sec:prop-ample}
  A $T$-Cartier divisor  $h = \sum h_P \in \tcadiv(\fan)$ is (semi-)ample if and only if all $h_P$ are strictly concave (concave) and $-h|_\sigma(0)$ is (semi-)ample for all tail cones $\sigma$, i.e. $\deg -h|_\sigma(0) = -\sum_P a^h_P (\sigma) > 0$ (or  a 
multiple of $-h|_\sigma(0)$ is principal).
\end{prop}

\begin{proof}
We first prove that semi-ampleness follows from the above criteria. Because $h$ is (strictly) concave the same is true for $h_0$. This implies that the $u^h(\sigma)$ are exactly the vertices of $\Box_{h}$ and $h^*(u^h(\sigma))=h|_\sigma(0)$.

  The semi-ampleness for $h^*(u),\; u \in \Box_h\cap M$ follows from the semi-ampleness at the vertices. Indeed if $D, D'$ are semi-ample divisors on $Y$ this is also true for $D + \lambda (D'-D)$ with $0\leq \lambda \leq 1$.

  Every vertex $(u,a_u)$ of $\graph_{h^*_P}$ corresponds to an affine piece of $h_P$ of the form $\langle u, \cdot \rangle-a_u$. If we let $f$ be such that $\divisor(f) = a_u P$ on $\loc \D$ for some $\D \in \mathcal{S}$ we then have $D_h|_{\tilde{\pdv}(\D)} = \divisor(f^{-1} \chi^{-u})$ (see \vref{sec:construction-cartier-div}).
  A point $(u,a_u)\in M \times \ZZ$ is a vertex of $h^*$ exactly if $(k u,k a_u)$ is a vertex of $(k \cdot h)^*$. Hence, after passing to a suitable multiple of $h$ we may assume, that $h^*(u)$ is base-point free with
  $f$ being a global section which generates $\CO(h^*(u))$ on $\loc \D$.
 Thus $f \chi^{u}$ is a global section of $\CO(D_h)$ which generates $\CO(D_h)|_{\tilde{\pdv}(\D)}$.

 To show the other direction, i.e. that semi-ampleness implies the above criteria, assume that $h_P$ is not concave. Then this is true also for every multiple of $\ell \cdot h_P$ and hence there is an affine piece $\langle u, \cdot \rangle-a_u$ of $\ell h_P$ such that $a_u > (\ell h_P)^*(u)$. This means there is no global section $f\chi^u$ such that $\divisor(f) = a_u P$. But this contradicts the base-point freeness of $D_{\ell h}$ and hence the semi-ampleness of $D_h$.


To get the statement for ampleness note that a support function $h$ on a polyhedral subdivision is strictly concave if and only if for every support function $h'$ there is a $k \gg 0$ such that $h'+kh$ is concave.
\end{proof}

\begin{cor}
  $\tilde{\pdv}(\fan)$ is projective if and only if all $\fan_P$ are regular subdivisions, i.e. admit a strictly convex support function.
\end{cor}

\begin{rem}
	We see from proposition \ref{sec:prop-ample} that for $h\in SF(\fan)$, if the  $T$-invariant divisor $D_h$ is semi-ample, the corresponding dual function $h^*$ is in fact a divisorial polytope. Conversely, if $h^*$ is a divisorial polytope, the associated divisor on the associated $T$-variety is semi-ample.
\end{rem}
\subsection{Intersection numbers}
\label{sec:intersection}

\begin{defn}
  For a divisorial polytope $h^*$ we define its {\em volume} to be
  $$\vol h^* := \sum_P \int_{\Box_h} h^*_P \vol_{M_\RR}$$

 For divisorial polytopes $h^*_1,\ldots, h^*_k$ we define their {\em mixed volume} by
  $$V(h^*_1,\ldots, h^*_k):= \sum_{i=1}^k (-1)^{i-1} \sum_{1\leq j_1 \leq \ldots j_i \leq  k} \vol(h^*_{j_1} + \cdots + h^*_{j_i})$$
\end{defn}

\begin{prop}\label{prop:intersection-numbers}
  Assume that on $X$ Kodaira's Vanishing Theorem holds.
  \begin{enumerate}
  \item  If $D_h$ is semi-ample, for the self-intersection number we get \label{item:prop-vol-vol}
    $$(D_h)^{(m+1)} = (m+1)!\vol h^*.$$
  \item  Let $h_1, \ldots, h_{m+1}$ define semi-ample divisors $D_i$ on $X(\fan)$. Then $$(D_1 \cdots D_{m+1}) = (m+1)!V(h^*_1, \ldots, h^*_{m+1}).$$\label{item:prop-vol-mixed}
  \end{enumerate}
\end{prop}
\begin{proof}
  If we apply (\ref{item:prop-vol-vol}) to every sum of divisors from $D_1,\ldots,D_{m+1}$ we get (\ref{item:prop-vol-mixed}) by the multi-linearity and symmetry of intersection numbers.
  
  To prove (\ref{item:prop-vol-vol}) we first recall that
$$(D_h)^{m+1} = \lim_{\nu \rightarrow \infty} \frac{(m+1)!}{\nu^{m+1}} \chi(X,\CO(\nu D_h)),$$
but for projective $X:=\pdv(\fan)$ and nef divisors the ranks of higher cohomology groups are asymptotically irrelevant \cite[Thm. 6.7.]{MR1919457} so we get
  $$(D_h)^{m+1} = \lim_{\nu \rightarrow \infty} \frac{(m+1)!}{\nu^{m+1}} h^0(X,\CO(\nu D_h)).$$

Note that $(\nu h)^*(u)=\nu \cdot h^*(\frac{1}{\nu}u)$. Now we can bound $h^0$ by
\begin{equation}
  \label{eq:h-bound}
\sum_{u \in \nu \Box_h \cap M} \left( \deg \lfloor \nu h^*\left({\textstyle \frac{1}{\nu}} u \right)\rfloor - g(Y) +1\right) \leq h^0(\CO(\nu D_h)) \leq \sum_{u \in \nu \Box_h \cap M} \deg \lfloor \nu h^*\left({\textstyle \frac{1}{\nu}}  u\right) \rfloor + 1.
\end{equation}

On the one hand we have
\begin{eqnarray*}
   \lim_{\nu \rightarrow \infty} \frac{(m+1)!}{\nu^{m+1}} \sum_{u \in \nu \Box_h \cap M} \deg \lfloor \nu h^*\left({\textstyle \frac{1}{\nu}} u\right) \rfloor
&=& \lim_{\nu \rightarrow \infty} \frac{(m+1)!}{\nu^{m}} \sum_{u \in \Box_h \cap \frac{1}{\nu}M} \frac{1}{\nu} \deg \lfloor \nu h^*(u)\rfloor \\
&=& (m+1)! \int_{\Box_h} h^* \vol_{M_\RR}.
\end{eqnarray*}
On the other hand, for any constant $c$ we have  
$$\lim_{\nu\rightarrow \infty} \frac{1}{\nu^{m+1}} \sum_{u \in \nu \Box_h \cap M} c = c\cdot \lim_{\nu\rightarrow \infty} \frac{\# (\nu\cdot \Box_h \cap M)}{\nu^{m+1}}=0.$$
Thus, if we pass to the limit in \eqref{eq:h-bound}, the term in the middle has to converge to $\vol h^*$.
\end{proof}

\begin{rem}
  The theorem allows us to compute intersection numbers in characteristic $0$ as well as on $T$-surfaces in positive characteristic because Kodaira's vanishing theorem holds in these cases. We believe that the theorem holds as well for positive characteristic in higher dimensions; work is being done to show that the vanishing theorem holds there.
\end{rem}

\begin{cor}\label{cor:num-equiv}
	Let $h\in SF(\fan)$ and let $C$ be any one-cycle rationally equivalent to the intersection of Cartier divisors, each of which can be expressed as an integer linear combination of semi-ample Cartier divisors. Then $D_{h}\cdot C$ is equal to $D_{h+P-Q}\cdot C$ for all points $P,Q\in Y$.
\end{cor}
\begin{proof}
	We have $$D_{h+P-Q}\cdot C=(D_h-D_{-P}+D_{-Q})\cdot C=D_h\cdot C-D_{-P}\cdot C+D_{-Q}\cdot C$$
	so it is sufficient to show that $D_{-P}\cdot C= D_{-Q} \cdot C$. Now, $D_{-P}$ and $D_{-Q}$ are semi-ample, so we can apply proposition \ref{prop:intersection-numbers}. Using the fact that $\vol( (-P)^*+\widetilde{h}^*)=\vol( (-Q)^*+\widetilde{h}^*)$ for all $\widetilde{h}\in SF(\fan)$ gives the desired equality.
\end{proof}

\begin{exmp}
 We know by proposition~\ref{sec:prop-ample} that $D_h$ in our threefold is ample. We have $\vol h^*= 21$. Hence, $X$ is Fano of degree $21$.
\end{exmp}

\subsection{Genus of Curves on Surfaces}
Let $X=\tilde{X}(\fan)$ be a two-dimensional $T$-variety and let $h\in SF(\fan)$ be a support function on $\fan$. For any curve $C\in |D_h|$, we show how to calculate the arithmetic genus $g(C)$. As a corollary, we can calculate the Euler characteristic $\chi(X,\CO(D_h))$ if $X$ is smooth.

\begin{defn}
For any $h\in SF(\fan)$, let $$\inn h_{P}^*:=\sum_{u\in \Box_h^\circ\cap M}
\#\{a\in \mathbb{Z}_{\geq 0} \mid a < |h_P^*(u)|\}\cdot \frac{h_P^*(u)}{|h_P^*(u)|}$$
for each point $P\in Y$, where $\Box_h^\circ$ is the interior of $\Box_h$. Furthermore, let $$\inn h^*:=\sum_{P \in Y} \inn h_P^*.$$
\end{defn}

\begin{defn}
For any $h\in SF(\fan)$, let
\begin{equation*}
\#h_P^*:=\sum_{u\in \Box_h\cap M} \lfloor h_P^*(u) \rfloor
\end{equation*}
for any point $P\in Y$ and let
\begin{equation*}
\#h^* := \sum_{u \in \Box_h\cap M} \deg \lfloor h^*(u) \rfloor=\sum_{Y\in P} \#h_P^*.
\end{equation*}
\end{defn}

\begin{rem}
    Note that $\inn h_P^*$ is the number of ``interior'' lattice points between the graph of $h_P^*$ and $0$ counted with their signs, where lattice points in height $0$ are counted as long as they aren't on the boundary of $\Box_h$. Similarly, if $\#h_P^*(h)\geq 0$ for all $u\in\Box_h$, $\#h_P^*$  is the sum of the number of lattice points between the graph of $\#h_P^*$ and $0$, where we count no lattice points in height $0$ but all lattice points lying on the graph of $h_P^*$.
\end{rem}

We will use the following lemma:

\begin{lem}\label{lemma:genpick}
With notation as above, $2\cdot\vol h_P^*=\inn h_P^*+ \# h_P^*$ for all $P\in Y$. It follows in particular that $2\cdot\vol h^*=\inn h^*+ \# h^*$.
\end{lem}

\begin{proof}
Fix some $P\in Y$. Suppose now that $h_P^*(u)\geq 0$ for all $u\in\Box_h$ and set 
\begin{equation*}
\Delta=\conv\left\{\left\{(u,h_P^*(u))\right\}\cup\left\{(u,0) \right\}\right\},
\end{equation*}
where $u\in \Box_h$. This is a convex polytope in $M_\mathbb{Q}'$, where $M'=M\times \mathbb{Z}$. Pick's theorem tells us that $2 \cdot \vol\Delta+2=\#(\Delta\cap M') + \# (\Delta^\circ\cap M')$. Now $\vol\Delta=\vol h_P^*$, $\#(\Delta\cap M)=\#h_P^*+\#(\Box_h\cap M)$, and $\#(\Delta^\circ\cap M)=\inn h_P^*-\#(\Box_h\cap M)+2$, so the desired equality follows. For general $h_P^*$, choose $j$ such that $\widetilde{h}_P^*(u):=h_P^*(u)+j\geq 0$ for all $u\in \Box_h$. Then $2\cdot\vol \widetilde{h}_P^*=\inn \widetilde{h}_P^*+ \# \widetilde{h}_P^*$ and for $j_P^*(u):=j$ we have $2\cdot\vol j_P^*=\inn j_P^*+ \# j_P^*$. Since $\vol$, $\inn$, and $\#$ are additive at least for integer-valued functions, the desired equality follows for $h_P^*=\widetilde{h}_P^*-j_P^*$. 
\end{proof}

We are now able to prove the following proposition:

\begin{prop}\label{prop:genusformula}
Let $h\in SF(\fan)$ be any support function such that $D_h$ is semi-ample. Then for $C\in |D_h|$, the arithmetic genus of $C$ is given by $$g(C)=\inn h^* + 1 + \vol \Box_h \cdot (g(Y)-1),$$ where $g(Y)$ is the genus of $Y$. 
\end{prop}
\begin{proof}
Without loss of generality, we can take the curve $C$ to equal $D_h$. Indeed, arithmetic genus is invariant under rational equivalence and since $|D_h|$ isn't empty, it must contain some $T$-invariant effective divisor. We compare the genus of $C$ with that of a comparable curve $C_0$ on $X_0:=Y\times \mathbb{P}^1$ and then compute the genus of $C_0$ directly. To begin with, note that we can find monoidal transformations $\pi_i:X_i\to X_{i-1}$ $1\leq i\leq k$ such that; 
\begin{enumerate}
\item $X_i$ is a $T$-variety ;
\item $\pi_i$ is $T$-equivariant;
\item There is a birational $T$-equivariant morphism $\varphi:X_k\to X$.
\end{enumerate}
This is done as follows: Let $\Sigma$ be the fan $\{\mathbb{Q}_{\geq 0},\mathbb{Q}_{\leq 0},\{0\}\}$\dots and let $\fan_P^0:=\Sigma$ for all points $P\in Y$. Then $X_0=\tilde{X}(\fan^0)$. Each morphism $\pi_i$ corresponds to an additional subdivision in the fan $\fan^{i-1}$ at exactly one point. Thus, we keep on refining until we get a $\fan^k$ which is a smooth common refinement of $\fan$ and $\fan^0$; this gives us our morphism $\varphi$. Finally, we let $\pi:X_k\to X_0$ be the composition of the $\pi_i$'s.

We now pull back $C$ to $C_k:=\varphi^*(C)$. Thus we now have $C_k=D_h$, where $h$ is now considered as a support function on $\fan^k$. Furthermore, this doesn't change the arithmetic genus, that is, $g(C)=g(C_k)$. Define now inductively $C_{i-1}={\pi_i}_*(C_i)$ for $1\leq i \leq k$. One easily checks that $C_0=D_{\widetilde{h}}$, where $\widetilde{h}\in SF(\fan^0)$ is the support function given by the divisorial polytope $\widetilde{h}_P^*:=\max_{u\in\Box_h} h_P^*(h)$ with $\Box_{\widetilde{h}}:=\Box_h$. Note that since $C$ is semi-ample, each $C_i$ is semi-ample as well. We will now calculate the difference between $g(C_k)$ and $g(C_0)$.

We first consider a special case, namely, suppose that $h_P^*$ is trivial everywhere except for at two points $Q_1\neq Q_2$. If $Y=\mathbb{P}^1$, all the varieties $X_i$ and $X$ are toric. In this case, the divisor $D_h$ can be understood in toric terms as the polytope 
\begin{equation*}
\Delta_h:=\conv \Gamma_{h_{Q_1}^*}\cup\Gamma_{-h_{Q_2}^*}
\end{equation*}
and $D_{\widetilde{h}}$ corresponds to $\Delta_{\widetilde{h}}$, which is defined in a similar manner. Then $$g(C_k)-g(C_0)=I(\Delta_h)-I(\Delta_{\widetilde{h}}),$$ where $I(\Delta)$ is the number of interior lattice points of $\Delta$, see for example \cite{MR2272243}, prop. 5.1. But we have $I(\Delta_h)=\inn h_{Q_1}^*+\inn h_{Q_2}^*-\#(\Box_h^\circ\cap M)$ and a similar equation for $\widetilde{h}$, which leads to
\begin{equation}
g(C_k)-g(C_0)=\inn h^*-\inn \widetilde{h}^*.\label{eqn:genusdiff}
\end{equation}

Now, equation \eqref{eqn:genusdiff} actually holds in general, not just in the toric case. To see this, note that for each $1\leq i\leq k$, $C_i=\pi_i^*(C_{i-1})+r_i\cdot E_i$, where $E_i$ is the exceptional divisor of $\pi_i$. Then similar to \cite{MR0463157}, V.3.7 we have $g(C_i)=g(C_{i-1})-\frac{1}{2}r_i(r_i+1)$. Thus, $$g(C_k)-g(C_0)=\sum_{i=1}^k -\frac{1}{2}r_i(r_i+1).$$ However, for each $1\leq i \leq k$, the integer $r_i$ can be determined combinatorially by comparing the polyhedral subdivisions $\fan_P^i$ and $\fan_P^{i-1}$ for the single point $P\in Y$ where these fansy divisors differ. Thus, the integers $r_i$ can be calculated exactly as if we were in the toric case, so we get $$\sum_{i=1}^k -\frac{1}{2}r_i(r_i+1)=\inn h^*-\inn \widetilde{h}^*.$$ Equation \eqref{eqn:genusdiff} follows.

We now calculate $g(C_0)$. From the adjunction formula, we have
\begin{equation*}
g(C_0)=\frac{D_{\widetilde{h}}^2+D_{\widetilde{h}}\cdot K_0}{2}+1
\end{equation*}
 for $K_0$ a canonical divisor on $X_0$, see \cite{MR0463157}, V.1.5. The theorem of Riemann-Roch for surfaces (\cite{MR0463157}, V.1.6) gives us
\begin{equation*}
\chi(X_0,\CO (D_{\widetilde{h}}))=\frac{D_{\widetilde{h}}^2-D_{\widetilde{h}}\cdot K_0}{2}+\chi(X_0,\CO_{X_0}).
\end{equation*}
Thus, 
\begin{equation*}
g(C_0)=D_{\widetilde{h}}^2+1+\chi(X_0,\CO_{X_0})-\chi(X_0,\CO (D_{\widetilde{h}})).
\end{equation*} 
Now, $\chi(X_0,\CO_{X_0})=1-g(Y)$ (see \cite{MR0463157}, V.2.5). Likewise, if $p:X_0\to Y$ is the projection, we have
\begin{align*}
\chi\left(X_0,\CO (D_{\widetilde{h}})\right)&=\chi\left(Y,p_*\CO (D_{\widetilde{h}})\right)\\
				&=\sum_{u\in\Box_{{h}}\cap M} \chi(Y,\CO(\widetilde{h}^*(u)))\\
&=\#\widetilde{h}+(1-g)\cdot(\vol\Box_{{h}}+1),
\end{align*}
where the last equation follows from Riemann-Roch for curves. We also have that $D_{\widetilde{h}}^2=2\cdot\vol \widetilde{h}$. Making these substitutions results in
\begin{align*}
g(C_0)&=2\cdot \vol \widetilde{h}+1+\vol\Box_{{h}}\cdot(g(Y)-1)-\#\widetilde{h}\\
&=\inn \widetilde{h}+1+\vol\Box_{{h}}\cdot(g(Y)-1),
\end{align*}
the second equality coming from lemma \ref{lemma:genpick}. Combining this with equation \eqref{eqn:genusdiff} completes the proof.
\end{proof}

\begin{cor}\label{cor:euler}
For any semi-ample $T$-invariant Cartier divisor $D_h$ on a smooth $T$-variety $X$, we have 
$$\chi(X,\CO(D_h))=\#h^*-(g(Y)-1)\cdot \#(\Box_h \cap M)=\sum_{u\in \Box_h\cap M}\chi(Y,\CO(h^*(u))).$$ 
\end{cor}
\begin{proof}
Using the adjunction formula and the Riemann-Roch theorem for surfaces as in the above theorem gives us the formula
\begin{equation*}
\chi\left(X,\CO (D_{{h}})\right)=D_{{h}}^2+1+\chi(X,\CO_{X})-g(C)
\end{equation*}
for some $C\in|D_h|$. We can use the above proposition to calculate $g(C)$. Combining this with the facts that $D_{{h}}^2=2\cdot\vol {{h}}$ and  $\chi(X,\CO_{X})=1-g(Y)$ along with lemma \ref{lemma:genpick} completes the proof of the first equality. The second equality follows directly from the theorem of Riemann-Roch for curves.
\end{proof}

At the and of this section we revisit our surface example and study all introduced concepts at it.
\begin{exmp}\label{ex:surface-div}
  \begin{figure}[htbp]
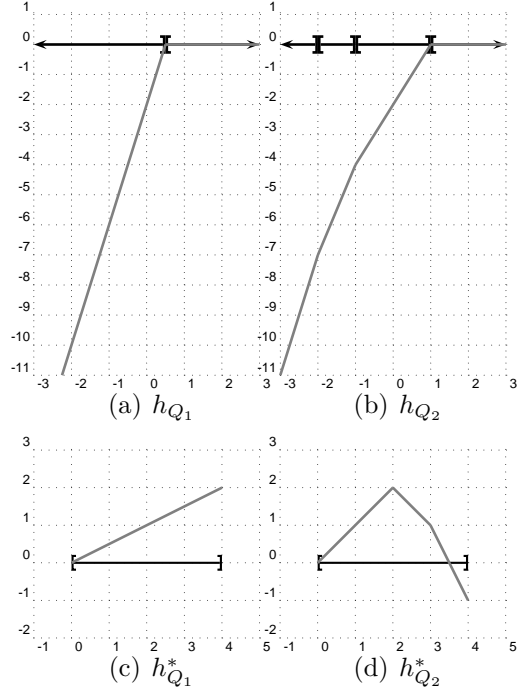

    \centering
    \subfigure[$h_{Q_1}$]{\exmpih}
    \subfigure[$h_{Q_2}$]{\exmpiih}\\
    \subfigure[$h^*_{Q_1}$]{\exmpip}
    \subfigure[$h^*_{Q_2}$]{\exmpiip}
  \caption{$h$ and $h^*$ for a $T$-surface}\label{fig:surface-divisor}
\end{figure}
 We look at the Cartier divisor $D_h$ on our surface example
  where $h_{Q_1}$ and $h_{Q_2}$ are given by the tropical polynomials \tropical{0+(-2)*x^4} and \tropical{0+(-2)*x^2+(-1)*x^3+1*x^4}, respectively. One easily sees that $\Box_h=[0,4]$, and that $h_{Q_1}^*$ and $h_{Q_1}^*$ respectively correspond to the tropical polynomials \tropical{x^{1/2}} and \tropical{x+4*x^{-1}+7*x^{-2}}. In other words, $h_{Q_1}^*(u)=u/2$ and 
\begin{equation*}
	h_{Q_2}^*(u)=\left\{\begin{array}{l@{\textrm{if }}l}
u\qquad&u\leq 2\\
4-u&2\leq u\leq 3\\
7-2u\qquad&u\geq 3.
	\end{array}\right.\\
\end{equation*}
In figure \ref{fig:surface-divisor} we sketch $h$ and the corresponding divisorial polytope $h^*$.

  We can use proposition~\ref{prop:cartier2weil} to compute the corresponding Weil divisor:  $4 D_{\QQ_{\leq 0}} + 4 D_{(Q_2,2)} + 7 D_{(Q_2,1)}$. $D_h$ is semi-ample, so by proposition~\ref{prop:intersection-numbers} we get $(D_h)^2=15$. Finally, from proposition \ref{prop:genusformula} we know that a section of $D_h$ has genus $5+4\cdot g(Y)$.

  We may also start with $h^*$ and take the dual $h$ to construct a fansy divisor as described above. We recover $\fan$ this way. $X:=\tilde{\pdv}(\fan)$ is not smooth, but a refinement of the polyhedral subdivisions (see figure~\ref{fig:refined-div}) gives a smooth surface $X'$ (this is will not be proved here; c.f. \cite{suess08}). Using corollary \ref{cor:euler}, we can calculate that $\chi(X',\CO(D_h))=12-5\cdot g(Y)$.

  \begin{figure}[hbtp]
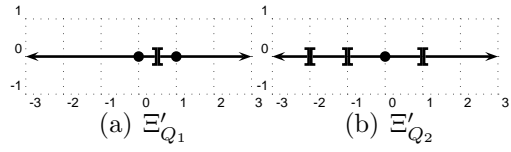

    \centering
       \subfigure[$\fan'_{Q_1}$]{\exmpismooth}
    \subfigure[$\fan'_{Q_2}$]{\exmpiismooth}
    \caption{A refined fansy divisor} \label{fig:refined-div}
  \end{figure}
\end{exmp}

\section{$T$-Codes and their Parameters}
\label{sec:t-codes}
\subsection{Construction}
\label{sec:codes-construction}
Let $Y$ be a curve over $\mathbb{F}_q$ and let $h^*$ be a divisorial polytope. Let  $\mathcal{P}=\{P_1,\ldots,P_l\}$ be some subset of the $\mathbb{F}_q$-rational points of $Y$ such that for $i=1,\ldots,l$, $h_{P_i}^*$ is affine and $h_{P_i}^*(u)\in \mathbb{Z}$ for $u\in \Box_h\cap M$.  Let $\fan$ be the fansy divisor associated to $h^*$ and let $\fan'$ be some minimal refinement such that $X:=\tilde{X}(\fan')$ is smooth. Note that for each point $P_i\in\mathcal{P}$, $\fan'_{P_i}=v(P_i)+\Sigma$, for a unique lattice point $v(P_i)$ and tail fan $\Sigma$. Set $m=\dim M$. For each point $P_i$ let $P_i^1,\ldots,P_i^{(q-1)^m}$ be the $(q-1)^m$ $\mathbb{F}_q$-rational points on $X$ of the open $T$-orbit contracting to $P_i$.

The support function $h$ associated to $h^*$  corresponds to a semi-ample $T$-invariant $\mathbb{F}_q$-rational Cartier divisor $D_h$ on $X$. We denote the corresponding line bundle by $\CO(D_h)$ and let $L(D_h)=\Gamma(X,\CO(D_h))$. For each point $P_i^j$ fix some isomorphism $\CO(D_h)_{P_i^j}\cong \mathbb{F}_q$. Consider the $\mathbb{F}_q$-linear map   
\begin{align*}
	\ev:L(D_h)&\to\mathbb{F}_q^{l(q-1)^m}\\
	f&\mapsto \left(f_{P_1^1},f_{P_1^2},\ldots,f_{P_l^{(q-1)^m}} \right)
\end{align*}
where $f_{P_i^j}$ is the image of $f$ in $\mathbb{F}_q$ following the identification with $\CO(D_h)_{P_i^j}$. In other words, the above map evaluates the rational function $f$ at the $l(q-1)^m$ points $P_i^j$ $1\leq i \leq l$,  $1\leq j \leq (q-1)^m$. The image of $\ev$ is a linear subspace of $\mathbb{F}_q^{l(q-1)^m}$ and thus a linear code of length $n=l(q-1)^m$; we denote it by $\mathcal{C}(Y,h^*,\mathcal{P})$. If $\mathcal{P}$ is maximal, we simply denote it by $\mathcal{C}(Y,h^*)$. Note that although $\mathcal{C}(Y,h^*,\mathcal{P})$ indeed depends on the way we identify $\CO(D_h)_{P_i^j}$ with $\mathbb{F}_q$, its length $n$, dimension $k$, and its minimum distance $d$ do not. Thus, we will always assume that some such isomorphisms are given, but will not concern ourselves further with them. 

\begin{rem}\label{rem:easy-construction}
If $h_{P_i}^*=0$ for $i=1,\ldots,l$, then $\mathcal{C}(Y,h^*,\mathcal{P})$ is equivalent as code to the image of the map
\begin{align*}
	\ev:\bigoplus_{u\in \Box_h\cap M} \Gamma\left(\CO(h^*(u))\right)\chi^u&\to\mathbb{F}_q^{l(q-1)^m}\\
	g\chi^u&\mapsto \left(g(P_1)\chi^u(Q_1),g(P_1)\chi^u(Q_2),\ldots,g(P_l)\chi^u(Q_{(q-1)^m})\right)
\end{align*}
where $Q_1,\ldots,Q_{(q-1)^m}$ are the $\mathbb{F}_q$-rational points of the $m$-dimensional torus. Thus, in this case the isomorphisms $\CO(D_h)_{P_i^j}\cong \mathbb{F}_q$ are not only irrelevant but also unnecessary. Now let $\mathcal{C}_u$ be the $[(q-1)^m,1,(q-1)^m]$ code generated by $\left(t^{u}\right)_{t \in {(\mathbb{F}_q^*)}^m}$ and let $\mathcal{C}(Y,h^*(u),\mathcal{P})$ be the AG code corresponding to the curve $Y$, divisor $h^*(u)$, and  point set $\mathcal{P}$. Then as mentioned in the introduction, we can also define $\mathcal{C}(Y,h^*,\mathcal{P})$ simply as
$$\mathcal{C}(Y,h^*,\mathcal{P})= \sum_{u \in \Box_h \cap M} \mathcal{C}_u \otimes \mathcal{C}(Y,h^*(u),\mathcal{P}).$$
\end{rem}

\subsection{Estimate on Dimension}
\label{sec:codes-dim}
Assume that the map $\ev$ is injective. This is always the case if the bound given below for the minimum distance is larger than zero. We then have that
\begin{equation*}
k=\dim_{\mathbb{F}_q} L(D_h).
\end{equation*}
Using proposition \ref{prop:global_sections}, we thus get that 
\begin{equation*}
k=\sum_{u\in \Box_h\cap M} \dim \Gamma (Y,\mathcal{O}(h^*(u))).
\end{equation*}
We can approximate $k$ using only the combinatorics of $h^*$. Let
\begin{equation*}
\gamma(u)=\left\{\begin{array}{l@{\qquad\textrm{if}\quad} l}
\deg \lfloor h^*(u) \rfloor +1 -g(Y) & \deg \lfloor h^*(u) \rfloor +1 -g(Y)>0\\
1 &\deg \lfloor h^*(u) \rfloor +1 -g(Y)\leq 0 \ \textrm{and}\  h^*(u)\geq 0\\
0 & \textrm{otherwise.} 
\end{array}\right.
\end{equation*}

\begin{prop}\label{prop:codedim}
If the evaluation map $\ev$ is injective, then
\begin{equation}\label{eqn:kbound}
\#h^*+\#(\Box_h\cap M)(1-g)\leq\sum_{u\in \Box_h\cap M} \gamma(u) \leq k\leq \#h^* +\#(\Box_h\cap M). 
\end{equation}
Furthermore, 
\begin{equation}\label{eqn:kequals}
k=\#h^*+\#(\Box_h\cap M)(1-g))
\end{equation} if $\deg h^*(u) >2 g(Y)-2$ for all $u\in\Box_h\cap M$.
\end{prop}
\begin{proof}
The leftmost inequality in \eqref{eqn:kbound} follows from the definition of $\gamma(u)$. We now consider the second inequality in \eqref{eqn:kbound}. Fix some degree $u\in \Box_h\cap M$. Then we always have $\dim \Gamma(Y,\CO(h^*(u)))\geq 0$, and if $h^*(u)$ is effective, then $\dim \Gamma(Y,\CO(h^*(u)))\geq 1$. Using the theorem of Riemann-Roch (see for example \cite{MR0463157}) we also have $\dim \Gamma(Y,\CO(h^*(u)))\geq \deg h^*(u)+1-g$ and the inequality follows. If $\deg h^*(u) >2 g(Y)-2$ then equality holds, so \eqref{eqn:kequals} follows. Finally, the right inequality in \eqref{eqn:kbound} follows from  $\dim \Gamma(Y,\CO(h^*(u)))\leq \deg h^*(u)+1$.
\end{proof}

\subsection{General Lower Bound on Minimum Distance}
\label{sec:codes-dist}
One strategy to get an estimate for $d$ is using techniques of intersection theory, as first presented in \cite{MR1866342}. These techniques have been applied to toric varieties, see for example \cite{MR1953195} and \cite{MR2360532}. We first consider the general case and then specialize to surfaces.

Let $e_1^*,\ldots,e_m^*$ be a basis for $M$. For $P\in\mathcal{P}$ and $\eta_1,\ldots,\eta_{m-1}\in\mathbb{F}_q^*$ define  $l(q-1)^{m-1}$ curves
\begin{equation*}
C_{P,\eta_1,\ldots,\eta_{m-1}}:=(P,v(P))\cap V\left(\{\chi^{e_i^*}-\eta_i\}_{i=1}^{m-1}\right).
\end{equation*}
Each point $P_i^j$ lies on exactly one of these curves. Furthermore, each curve $C_{P,\eta_1,\ldots,\eta_{m-1}}$ is rationally equivalent to  
\begin{equation*}
	C_P:=(P,v(P))\cap V\left(\{\chi^{e_i^*}\}_{i=1}^{m-1}\right)=D_{0-P}\cdot {(D_{-e_1^*})}_{\geq 0}\cdot\ldots\cdot{(D_{-e_{m-1}^*})}_{\geq 0}
\end{equation*}
where the second equality follows from proposition \ref{prop:cartier2weil}, $e_i^*$ is considered as an element of $SF(\fan)$, and ${(D_{-e_i^*})}_{\geq 0}$ is the effective part of $D_{-e_i^*}$.

Fix some section $s\in L(D_h)$; this corresponds to an effective divisor $(s)_0=D_h+(s)$. By $Z(s)$ we denote the number of points $P_i^j$ such that $s_{P_i^j}=0$. Equivalently, $Z(s)$ is the number of points $P_i^j$ contained in the support of $(s)_0$. Thus, one has the following lower bound for the minimum distance:
\begin{equation*}
d\geq l(q-1)^m-\max_{s\in L(D_h)} Z(s). 
\end{equation*}

Let $(s)_0$ vanish on exactly $\lambda$ of the curves $\{C_{P,\eta_1,\ldots,\eta_{m-1}}\}$. Following \cite{MR1866342} and setting $C=C_{P}$ for some $P\in\mathcal{P}$ we then have that 
\begin{equation}\label{eqn:zs-bad}
Z(s)\leq \lambda(q-1) + (l-\lambda)D_h\cdot C
\end{equation}
since $(s)_0\sim D_h$ and it follows from corollary \ref{cor:num-equiv} that $D_h\cdot C=D_h\cdot C_{P_i}=D_h\cdot C_{P_i,\eta_1,\ldots,\eta_{m-1}}$ for all $1\leq i\leq l$. Assuming that Kodaira's vanishing theorems holds on $X$, we can use proposition \ref{prop:intersection-numbers} to calculate $D_h\cdot C$. 

We now bound $\lambda$ in a method similar to \cite{MR2360532}. For the divisorial polytope $h^*:\Box_h\to \wdiv_\QQ Y$ let $\pr(\Box_h)$ be the projection of $\Box_h$ to $M/\mathbb{Z}e_m^*$ and define $\pr(h^*):\pr(\Box_h)\to\wdiv_\QQ (Y)$ by 
\begin{equation*}
\pr(h^*)_P(u)=\max_{(u,u_m)\in\Box_h\cap M} h_P^*((u,u_m)).
\end{equation*}
One easily checks that $\pr(h^*)$ is a divisorial polytope.
Assume that $\Box_h\subset \widetilde{u}+\{u\in M | 0\leq u_i \leq q-2\}$ for some $\widetilde{u}=(\widetilde{u}_1,\ldots,\widetilde{u}_m)\in M$. This also then holds for $\pr(\Box_h)$. We can write
\begin{equation*}
s=\chi^{\widetilde{u}_me_m^*}\cdot\left(s_0+s_1\chi^{e_m^*}+s_{q-2}\chi^{(q-2)e_m^*}\right)
\end{equation*}
where $s_i\in K(Y)(\chi^{u_1},\ldots,\chi^{u_{m-1}})$. In fact, one easily checks that $s_i\in L(D_{\pr(h)})$, where $D_{\pr(h)}$ is the $T$-invariant Cartier divisor on the $m$-dimensional $T$-variety $X_{\pr( h^* )}$ over $Y$ both determined by $\pr (h^* )$. If we restrict $s\cdot \chi^{-\widetilde{u}_me_m^*}$ to some curve $C_{P,\eta_1,\ldots,\eta_{m-1}}$ we get a polynomial $\overline{s}=\overline{s}_0+\overline{s}_1\chi^{e_m^*}+\overline{s}_{q-2}\chi^{(q-2)e_m^*}\in \mathbb{F}_q[\chi^{e_m}]$ of degree less than or equal to $q-2$. If $C_{P,\eta_1,\ldots,\eta_{m-1}}$ is a curve where $s$ vanishes, then $\overline{s}$ has $q-1$ zeros, so $\overline{s}\equiv 0$ and $\overline{s_i}=0$ for $0\leq i\leq q-2$. Thus the section $s_i \in L(D_{\pr(h)})$ vanishes on the point of $X_{\pr( h^* )}$ corresponding to the tuple $({P,\eta_1,\ldots,\eta_{m-1}})$. It follows that 
\begin{equation*}
\lambda \leq \max_{t\in L(D_{\pr(h)})} Z(t).
\end{equation*}
Thus, we can recursively bound $\lambda$ until $\dim(X)=2$.

\subsection{Lower Bound on Minimum Distance for $\dim(X)=2$}
We can provide a much better bound for $Z(s)$ when $X$ is a surface. Consider a global section $s$ of $\CO(D_h)$ as before such that $(s)_0$ vanishes on exactly $\lambda$ of the curves $\{C_{P_i}\}$, say $C_{Q_1},\ldots,C_{Q_\lambda}$ where the $Q_i$ are distinct points in $\mathcal{P}$. Thus, $s\in L(D_{\widetilde{h}})$, where $\widetilde{h}=h+\sum_{i=1}^\lambda Q_i$. Since $\widetilde{h}$ and $\sum_{i=1}^\lambda (-Q_i)$ are concave, it follows that $h^*=\widetilde{h}^*+(\sum_{i=1}^\lambda (-Q_i))^*$. In particular, we have that
\begin{equation*}
	\deg\widetilde{h}^*(u)=\deg h^*(u)-\lambda.
\end{equation*} 
Thus, $s$ can only have support in the weights $u\in\Box_{(h,\lambda)}$,  where \begin{equation*}
\Box_{(h,\lambda)}=\left\{u\in\Box_h\cap M |\deg \lfloor h^*(u) \rfloor \geq \lambda\right\}.
\end{equation*} 
It follows immediately that
\begin{equation*}
\lambda \leq \max_{u\in\Box_h\cap M} \deg \lfloor h^*(u) \rfloor :=\lambda_{0}. 
\end{equation*}

Having found a good bound for $\lambda$, we now try to improve on the upper bound for $Z(s)$ in equation \eqref{eqn:zs-bad}. By choosing a generator we can identify the lattice $N$ with $\mathbb{Z}$. Then $\sigma_{-}:=\mathbb{Q}_{\leq0}$ and $\sigma_{+}:=\mathbb{Q}_{\geq 0}$ are the two rays in $\Sigma$. Each of these rays corresponds to a $T$-invariant divisor. Let $\mu_{-}$ and $\mu_{+}$ respectively be the  coefficients of the prime divisors ${\sigma_-}$ and ${\sigma_+}$ in $(s)_0$. We want to find a lower bound for the sum $\mu_{-}+\mu_+$. This is easy if $s$ has support only in a single weight $u$, say $s=f\cdot \chi^u$: In this case, $(s)$ is $T$-invariant corresponding to the support function $-u-\divisor(f)$ and thus $\mu_{-}+\mu_+=-h_0(-1)-h_0(1)$ using proposition \ref{prop:cartier2weil}.

Let $u_{\min}$ and $u_{\max}$ be respectively the smallest and the largest weights in which $s$ has non-trivial support and let $\nu=u_{\max}-u_{\min}$. Note that we can bound $\nu$ by
\begin{equation*}
\nu\leq \nu(\lambda):=\max \Box_{(h,\lambda)}-\min \Box_{(h,\lambda)}.
\end{equation*}
Let $\mathcal{S}$ be some set of polyhedral divisors corresponding to some open covering of $X$ and consider some polyhedral divisor $\D\in\mathcal{S}$.
Now, the divisor $\sigma_-$ or $\sigma_+$ is contained in $\tilde{X}(\D)$ if and only if $\D$ has respectively $\sigma_-$ or $\sigma_+$ as tail cone. If the tail cone of $\D$ is $\sigma_+$, we can write 
\begin{equation*}
s=\chi^{u_{\min}} f^{-1} \cdot (s_0+s_1\chi+\ldots+s_\nu\chi^\nu)
\end{equation*}
 with $f,s_0,\ldots,s_\nu \in\CO(\loc\D)$
and so $(s)$ is the sum of some effective divisor and the $T$-invariant principal divisor $(f^{-1}\cdot\chi^{u_{\min}})$. Thus, using proposition \ref{prop:cartier2weil}, we have $\mu_+\geq -h_0(1)+u_{\min}$. On the other hand, if the tail cone of $\D$ is $\sigma_-$, we can write 
\begin{equation*}
s=\chi^{u_{\max}} f^{-1} \cdot (s_0\chi^{-\nu}+s_1\chi^{-\nu+1}+\ldots+s_\nu)
\end{equation*}
 with $f,s_0,\ldots,s_\nu \in\CO(\loc\D)$. Thus, using proposition \ref{prop:cartier2weil} again, we have $\mu_-\geq -h_0(-1)-u_{\max}$. Combining these two inequalities gives us
\begin{equation*}
\mu_{-}+\mu_+\geq \vol\Box_h -\nu\geq\vol\Box_h -\nu(\lambda),
\end{equation*}
where we use the easily checked fact that $-h_0(-1)-h_0(1)=\vol\Box_h$.

Now, each curve $C_P$ intersects with $\sigma_+$ in one point; similarly, $C_P$ and $\sigma_-$ intersect in some other point. Neither of these points is one of the points $P_i^j$ at which we are evaluating our section $s$. This means that for each of the $l-\lambda$ curves where we calculate the number of zeros of $(s)_0$ using intersection numbers, we have counted at least $\mu_-+\mu_+$ too many points. Furthermore, we can use proposition \ref{prop:intersection-numbers} to calculate that $D_h\cdot C=\vol\Box_h$. Thus, we can improve equation \eqref{eqn:zs-bad} to
\begin{equation*}
Z(s)\leq \lambda(q-1) + (l-\lambda)\nu(\lambda).
\end{equation*}
 
Summing up the results obtained here leads to the following:

\begin{prop}\label{prop:surfaceparam}
Let $\mathcal{C}(Y,h^*,\mathcal{P})$ be a toric code on a two-dimensional $T$-variety. Then the minimum distance of this code is bounded from below by
\begin{equation*}
d\geq \min_{0\leq \lambda\leq \lambda_{0} }\left[(l-\lambda)(q-1-\nu(\lambda))\right].
\end{equation*}
\end{prop}

\begin{rem}
In the literature concerning toric surface codes, the estimate for the minimum distance often contains a term involving the self-intersection number of one of the curves $C_P$. In our case, this term does not help since $C_P^2=0$, which can be easily seen using proposition \ref{prop:intersection-numbers}. However, the correction we make using $\mu_+$ and $\mu_-$ has a similar effect.
\end{rem}

\subsection{Upper Bound on Minimum Distance}
A simple upper bound on the minimum distance of a toric code is given in  \cite{MR2360532}. We adapt this to the case of $T$-varieties. This then gives us a way of testing if the lower bound on minimum distance attained above is sharp:

\begin{prop}\label{prop:kupper}
	Let $f\in K(Y)$ be such that $f\cdot \chi^u\in L(D_h)$ for all $u\in B\cap M $, where $B$ is lattice isomorphic to a lattice hyper-rectangle with side lengths $r_1,\ldots,r_m$, $r_i\leq q-1$. Furthermore, suppose that $f$ vanishes at $r_0$ of the points $P_i\in\mathcal{P}$. Then 
\begin{equation}\label{eqn:kupper} d\leq (l-r_0)\cdot \left( (q-1)^m+\sum_{j=1}^m(-1)^{j}\sum_{i_1<\ldots<i_j}r_{i_1}\cdots r_{i_j}(q-1)^{m-j}\right). \end{equation} 
In particular, for $m=1$ we have $d\leq l(q-1)-r_1l-r_0(q-1)+r_0r_1$. 
\end{prop}
\begin{proof}
Choose a basis $e_1^*,\ldots,e_m^*$ of the lattice $M$ such that $B=\widetilde{u}+\prod_{i=1}^m [0,r_i]$. Let $\mathbb{F}_q^*=\{\eta_1,\ldots,\eta_{q-1}\}$. Now consider the rational function 
\begin{equation*}
f':=f\cdot\chi^{\widetilde{u}}\cdot\prod_{i=1}^m \prod_{j=1}^{r_i} (\chi^{e_i^*}-\eta_j).
\end{equation*}
One easily checks that $f'\in L(D_h)$. On the other hand, using inclusion-exclusion one sees that for each point $P_i\in Y$, $f'$ vanishes on 
\begin{equation*}
\sum_{j=1}^m(-1)^{j+1}\sum_{i_1<\ldots<i_j}r_{i_1}\cdots r_{i_j}(q-1)^{m-j} \end{equation*}
rational points of the open $T$-orbit contracting to $P_i$. The function $f'$ vanishes entirely on $r_0$ of these orbits, each of which has $(q-1)^m$ relevant points. Using inclusion-exclusion again and subtracting the total number of points on which $f'$ vanishes from the length $n=l(q-1)^m$ yields the desired result.
\end{proof}

As a consequence of the above proposition we get the following corollary:
\begin{cor}\label{cor:kupper}
Let $B\subset \Box_h$ be lattice isomorphic to a lattice hyper-rectangle with side lengths $r_1,\ldots,r_m$, $r_i\leq q-1$. Furthermore, for each $Q_j\in Y(\mathbb{F}_q)$ let $c_j\in\mathbb{Z}$ be such that $h_{Q_j}^*(u)\geq b_j$ for all $u\in \Box_h \cap M$. Inequality \eqref{eqn:kupper} then holds for $r_0:=\left(\sum c_j\right)-g(Y)$.  
\end{cor}
\begin{proof}
Using the above proposition, we just need to find $f\in K(Y)$ such that $f\cdot \chi^u\in L(D_h)$ for all $u\in B\cap M $ and such that $f$ vanishes at $r_0$ of the points $P_i\in \mathcal{P}$. Now, for any $r_0$ points $P_1,\ldots,P_{r_0}\in Y(\mathbb{F}_q)$, the divisor $D:=\sum c_j Q_j$ on $Y$ has a global section $f$ which vanishes on all $r_0$ points. Indeed, by Riemann-Roch $$\dim L(D-\sum_{i=1}^{r_0} P_i)= \left(\sum c_j\right) - r_0 +1-g(Y)=1.$$ Now since $h_{Q_j}^*(u)\geq c_j$ for all $u\in \Box_h \cap M$, $L(D)\subset L(h^*(u))$ for all $u\in B\cap M$ and we have found $f$ as desired.
\end{proof}

\begin{rem}
In the case of a toric code, the above corollary gives exactly the upper bound of \cite{MR2360532}.
\end{rem}

\section{Examples}
\label{sec:examples}
\subsection{Ruled Surfaces from Decomposable Vector Bundles}
\label{sec:ruled-surfaces}
Codes on ruled surfaces, or equivalently $\mathbb{P}^1$-bundles over a curve $Y$, were first considered in \cite{MR1866342}, where formulas for $n$ and $k$ and a lower bound for $d$ are given; global sections of some line bundle on $X$ are evaluated at all $\mathbb{F}_q$-rational points. This was then applied in \cite{lomont:2003} to surfaces of the form $X=\Proj (\CO_Y\oplus \CO_Y(-e))$. Assuming that the lower bound attained for $d$ there is sharp, the resulting codes are never better than a product code coming from a Reed-Solomon and Goppa code. However, by restricting the points at which we evaluate to a smaller set, better codes can be found. Indeed, consider the case $Y=\mathbb{P}^1$, $e>0$, where the resulting surface is the Hirzebruch surface $\mathcal{H}_e$, a toric variety. Codes obtained by evaluation on the points of the torus were considered in \cite{MR1953195}, with parameters considerably better than those of product codes. We wish to generalize this to bundles over curves of higher genus.

Consider the rank two locally free sheaf
\begin{equation*}
\mathcal{E}=\CO_Y\oplus\CO\left(\sum_{Q_i\in \mathbb{F}_q(Y)} \alpha_i Q_i\right)
\end{equation*}
for $\alpha_i\in\mathbb{Z}$ and set $X=\Proj (\mathcal{E})$. Any ruled surface coming from a decomposable vector bundle is isomorphic to such a $X$. Furthermore, $X$ can easily be described as a $T$-variety. Let $\Sigma\subset \mathbb{Q}$ be the fan consisting of the cones $\mathbb{Q}_{\leq 0}$, $\mathbb{Q}_{\geq 0}$, and $\{0\}$, and let $\fan$ be the fansy divisor with $\fan_{Q_i}=\alpha_i+\Sigma$. Then one can easily confirm that $X=\tilde{X}(\fan)$. We set $\alpha=\sum \alpha_i$.

Consider now any semi-ample $T$-invariant Cartier divisor $D_h$ on $X$. Then $h_0$ is of the form
\begin{equation*}
	h_0(v)=\left\{\begin{array}{l@{\textrm{if }}l}
u_{\max}\cdot v\qquad&v\leq 0\\
u_{\min}\cdot v& v\geq 0\\
	\end{array}\right.
\end{equation*} 
for some $u_{\min},u_{\max}\in\mathbb{Z}$ with $a:=u_{\max}-u_{\min}\geq 0$. It follows that $\Box_h=[u_{\min},u_{\max}]$. Furthermore, for each $Q_i\in\mathbb{F}_q(Y)$, $h_{Q_i}$ is of the form $h_{Q_i}(v)=h_0(v-\alpha_i)-b_i$ for some $b_i\in\mathbb{Z}$. Thus $h_{Q_i}^*(u)=\alpha_i\cdot u + b_i$. It follows that $\deg h^*(u)=\alpha\cdot u + b$.

\begin{figure}[htbp]
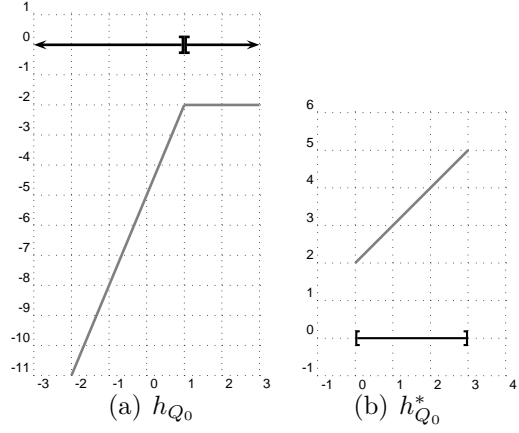

    \centering
    \subfigure[$h_{Q_0}$]{\exrsh}
    \subfigure[$h_{Q_0}^*$]{\exrsp}\\
  \caption{$h$ and $h^*$ for a simple ruled surface}\label{fig:ruled-surface}
\end{figure}

As an example, by setting $u_{\min}=0$, $u_{\max}=3$, $\alpha_0=1$, $b_0=2$, and all other possible parameters to $0$, we get the ruled surface with $h$ and $h^*$ as pictured in figure \ref{fig:ruled-surface}.

We now consider the code $\mathcal{C}(Y,h^*,\mathcal{P})$ for any set $\mathcal{P}$ of $\mathbb{F}_q$-rational points on $Y$; note that $h_P^*$ is affine and integer-valued on lattice points for any point $P\in \mathcal{P}$ as required. Set $l=\# \mathcal{P}$. For the sake of simplicity we shall assume that $u_{\min}=0$, $\alpha>0$ and $\alpha_i,b_i\geq 0$. This ensures that $h$ is in fact semi-ample, i.e. that $h^*$ is a divisorial polytope.  One easily confirms that $\lambda_0=b+a\cdot \alpha$ and that 
\begin{equation*}
	\nu(\lambda)=\left\{\begin{array}{l@{\textrm{if }}l}
a \qquad&\lambda\leq b\\
\lfloor a-\frac{\lambda-b}{\alpha}\rfloor\quad& \lambda\geq b.\\
	\end{array}\right.
\end{equation*}
Using proposition \ref{prop:surfaceparam} we then have that
\begin{equation*}
d\geq \min \left\{(l-b-a\cdot\alpha)(q-1)),(l-b)(q-1-a)\right\}.
\end{equation*}
We can then use corollary \ref{cor:kupper} to bound $d$ from above. Indeed, for $t\in \mathbb{Z}$, $0\leq t\leq a$ we have that $h_{Q_j}^*(u)\geq b_j+\alpha_j t$ for all $t\leq u \leq a$. Using the particular cases $t=0$ and $t=a$ results in the bound
\begin{equation*}
d\leq \min \left\{(l-b-a\cdot\alpha+g(Y))(q-1)),(l-b+g(Y))(q-1-a)\right\}.
\end{equation*}
Thus, we have upper and lower bounds for $d$ differing by at most $g(Y)\cdot(q-1)$.

We now use proposition \ref{prop:codedim} to find a lower bound for $k$. We always have that
\begin{equation*}
k\geq (a+1)(b+1+\alpha\cdot a/2-g(Y))
\end{equation*}
where equality holds if $b>g(Y)-2$.
Suppose now that $b\leq g(Y)$; set $c=\lceil ({g(Y)-b})/{\alpha} \rceil$. Now $h^*(u)$ is effective for every $u\in\Box_h\cap M$, so we can improve the bound on $k$ to 
\begin{equation}
k \geq (a+1-c)(b+1+\frac{\alpha}{2}(c+a)-g(Y))+c.\label{eqn:ellipticdim}
\end{equation} 
Note that equality holds if $g(Y)\leq 1$.

\begin{rem}
	In the case $Y=\mathbb{P}^1$ and $\alpha_i,b_i=0$ for all points $Q_i$ with the exception of some point $Q_0$, $X$ is the Hirzebruch surface $\mathcal{H}_\alpha$. If we set $\mathcal{P}=\mathbb{F}_q^*$, we recover the results of \cite{MR1953195}. Note that the curves we use to cover the points of the torus are perpendicular to those used by Hansen. In our case, these curves have self-intersection zero, but the adjustment we make with $\mu_-$ and $\mu_+$ compensates for this.
\end{rem}

We now compare these codes to product codes coming from a length $q-1$ Reed-Solomon and a Goppa code. A Reed-Solomon code has parameters $[q-1,k_1,d_1]$ with $d_1=q-k_1$ and $k_1\leq q-1$. Assume $\tau\in\mathbb{N}$ with $\tau > g(Y)-1$. Then the Goppa code on $Y$ gotten by evaluating a divisor $D$ of degree $\tau$  at $l$ rational points has parameters $[l,k_2,d_2]$ with $k_2\geq\tau-g(Y)+1$  and $d_2\geq l-\tau$, see for example (\cite{MR1667936}, vol. I ch. 10). The resulting product code $\mathcal{C}_{prod}$ has parameters $[l(q-1),k_1k_2,d_1d_2]$. For the product code we thus have the estimates
\begin{align*}
k_{prod}\geq k_{est}:=k_1(\tau-g(Y)+1),\\
d_{prod}\geq d_{est}:=(q-k_1)(l-\tau).
\end{align*}
We can then show the following:
\begin{prop}
	Fix some curve $Y$ and assume that $l\geq q+g(Y)-1$. Using notation as above, we can find $h^*$ and $\mathcal{P}$ as above such that the estimated parameters for $\mathcal{C}(Y,h^*,\mathcal{P})$ are better than those for $\mathcal{C}_{prod}$. Specifically, we show that
\begin{align}
k_{est}\leq (a+1)(b+1+\alpha\cdot a/2-g(Y)),\label{eqn:keq}\\
d_{est}< \min \left\{(l-b-a\cdot\alpha)(q-1)),(l-b)(q-1-a)\right\}.\label{eqn:deq}
\end{align}
\end{prop}
\begin{proof}
	First, suppose that $\tau\geq(k_1-1)$. We then set $a=k_1-1$ and choose some $\alpha\in\mathbb{N}$ such that $\alpha(k_1-1) \leq 2\tau$ and $\alpha(k_1-1)$ is divisible by two. Choose $b_i\geq0$ such that $b=\tau-\alpha(k_1-1)/2$. Choose any set $\mathcal{P}$ consisting of $l$ points. Equality in \eqref{eqn:keq} follows immediately and a quick calculation shows that \eqref{eqn:deq} holds as well.

Suppose instead that $\tau<(k_1-1)$. Set $\widetilde{k}_1=\tau-(g(Y)-1)$ and $\widetilde{\tau}=k_1+(g(Y)-1)$. Consider then the product code $\widetilde{\mathcal{C}}_{prod}$ obtained as product of the $\widetilde{k}_1$-dimensional Reed-Solomon code and the Goppa code corresponding to the divisor $\widetilde{\tau}Q_0$. Then one easily confirms that the estimated minimum distance and dimension for $\widetilde{\mathcal{C}}_{prod}$ are greater than or equal to those of $\mathcal{C}_{prod}$ and that $\widetilde{\tau}\geq(\widetilde{k}_1-1)$. Thus, we reduce to the first case above.
\end{proof}

\subsection{A Code on an Elliptic Curve}
The following example illustrates techniques that can be used to refine our estimate for minimum distance. It also demonstrates that there are $T$-codes with better parameters than the those estimated in the previous example. Before we begin, we first note the following lemma:
\begin{lem}\label{lemma:section-decomp}
Let $D_h$ be a $T$-invariant divisor on $\tilde{X}(\fan)$, and let $s$ be a section such that $(s)_0$ is not irreducible. Then we can find functions $h_1,h_2\in SF(\fan)$ and $s_1\in L(D_{h_1})$, $s_1\in L(D_{h_1})$ such that: 
\begin{enumerate}
\item $D_h=D_{h_1}+D_{h_2}$;
\item $(s)=(s_1)+(s_2)$;
\item $D_{h_i}$ is not rationally equivalent to $0$ for $i=1,2$.
\end{enumerate} 
\end{lem}
\begin{proof}
Since $(s)_0$ is not irreducible, we can write it as the sum of two nontrivial effective divisors $(s)_0=C_1+C_2$. Since the Picard group is generated by $T$-invariant divisors, we can find $h_1',h_2 '\in SF(\fan)$ such that $C_i=D_{h_i '}+(s_i')$ for some $s_i'\in L(D_{h_i'})$, $i=1,2$. We thus have $$D_h+(s)=D_{h_1 '}+(s_1')+D_{h_2 '}+(s_2').$$ Now set $s_1:=s_1'$, $h_1:=h_1'$, and $s_2:=s/s_1$, and let $h_2$ be the support function corresponding to the $T$-invariant divisor $D_{h_2'}+(s_2')-(s_2)$. These support functions and sections clearly fulfill the desired conditions.
\end{proof}

We now return to the divisor on the $T$-surface considered in example \ref{ex:surface-div}. For $Y$ either $\mathbb{P}^1$ or elliptic,  we have already noted that $D_h$ is semi-ample; this is the same as saying that $h^*$ is a divisorial polytope. Now if $Y=\mathbb{P}^1$ and $Q_1=0$, $Q_2=\infty$, the $T$-variety associated to $h^*$ is in fact toric and $h^*$ corresponds to the polytope in $\mathbb{Z}^2$ given by $\conv\{(0,0),(2,-2),(3,-1),(4,1),(4,2)\}$. Let $\mathcal{P}=Y\setminus\{Q_1,Q_2\}$; the example of $\mathcal{C}(\mathbb{P}^1,h^*,\mathcal{P})$ is considered in \cite{soprunov:2008}, where it is shown using the Hasse-Weil bound that $d\geq (q-1)^2-3(q-1)-2\sqrt 2 +1$ for all $q\geq 19$. We now calculate the parameters $d$ and $k$ for $\mathcal{C}(Y,h^*,\mathcal{P})$ in the case that $Y$ is an elliptic curve.

In calculating $k$, note that $\deg h^*(u)>0$ for $u>0$. Thus, in these degrees we have that $\dim L(D_h)_u=\deg h^*(u)$. On the other hand $h^*(0)=0$ which is effective, so $\dim L(D_h)_0=1$. Adding everything up we get that $k=8$.

Proposition \ref{prop:kupper} gives us an easy upper bound for $d$. If we set $f:=1$, we have that $f\cdot\chi^u\in L(D_h)$ for $u\in{0,1,2,3}$. Indeed, $h^*(u)$ is effective in these degrees. Thus, it follows that $d\leq l(q-1)-3l$.

We now bound $d$ from below. One easily checks that $\lambda_0=3$. Likewise, one can easily calculate that $\nu(0)=4$, $\nu(1)=3$, $\nu(2)=1$, and $\nu(3)=0$. Now consider some section $s$ such that $\lambda=1$. We claim that we actually must have that $\nu\leq 2$. The section $s$ cannot have support in weight $0$ since $\deg h^*(0)-1=-1$. Furthermore, $s$ cannot have support in weight $1$. Indeed, $\Gamma(Y,\CO(Q_2-P))=0$ for any point $P\neq Q_2$, since $Y\neq \mathbb{P}^1$. It follows that for any section $s$ with $\lambda\neq 0$ or with $\lambda=0$ and $\nu<4$ we have $Z(s)\leq \lambda(q-1)+l(3-\lambda)$; if we assume that $l\geq q-1$, it follows that $Z(s)\leq 3l$

Now consider some section $s$ such that $\lambda=0$ and $\nu=4$; we will show that under certain assumptions we also have $Z(s)\leq 3l$. First, suppose that $(s)_0$ is irreducible. Then using the Hasse-Weil bound for singular curves as stated in \cite{MR1394921}, we have that the number $\#(s)_0(\mathbb{F}_q)$ of $\mathbb{F}_q$-rational points on $(s)_0$ is bounded above by
\begin{equation*}
\#(s)_0(\mathbb{F}_q)\leq q+1 + 2 g\sqrt q
\end{equation*}
where $g:=g( (s)_0)$ is the arithmetic genus of $(s)_0$. Note that this only depends on the divisor $D_h$ and not on $s$. Now, if we require that 
\begin{equation*}
	q\geq \left(\frac{g+\sqrt{g^2+8}}{2}\right)^2
\end{equation*}
it follows that 
\begin{equation*}
Z(s)\leq q+1+2g\sqrt q \leq (q-1)3.
\end{equation*}
In our case, it follows from proposition \ref{prop:genusformula} that $g=9$ so the required bound on $q$ is $q\geq 89$.

Suppose on the other hand that $(s)_0$ is not irreducible. Let $h_1,h_2\in SF(\fan)$ be support functions and $s_i\in L(D_i)$ $i=1,2$ sections as in lemma \ref{lemma:section-decomp}, ordered such that $\vol \Box_{h_1} \leq \vol \Box_{h_2}$. It easily follows that $\nu(s)=\nu(s_1)+\nu(s_2)$ and by remark \ref{rem:h*sum} we have $h^*\geq h_1^*+h_2^*$. Now if $s_1$ only has support in a single degree, $(s_1)_0$ is $T$-invariant. Thus we have $Z(s_1)=0$ and $Z(s)=Z(s_2)$. 
Indeed, since $\lambda=0$, $(s_1)_0$ cannot contain one of the curves $C_P$ covering the points of evaluation, and all other $T$-invariant prime divisors don't contain any points of evaluation.
Now note that $h_2^*\leq h^*+(f)$ for some $f\in K(Y)$. Thus, $g\left( (s_2)_0\right)\leq g\left( (s)_0\right)$ and if $(s_2)_0$ is irreducible, the above argument with the Hasse-Weil bound gives the desired bound. If not, we replace $h$ and $s$ by $h_2$ and $s_2$ and repeat the process until we have an irreducible section and thus the desired bound, or have sections $s_1'$ and $s_2'$ both with support in multiple weights.  

We have now reduced to the situation where $h'\in SF(\fan)$ with $s'\in L(D_{h'})$, ${h'}^*\leq h^*+(f)$, for this $s'$ we have $\nu=4$, and $h'$ and $s'$ admit a decomposition into $h_1',h_2'$ and $s_1',s_2'$ such as in lemma \ref{lemma:section-decomp} such that both sections have support in multiple weights. We show that this is impossible. We first note that since $\nu=4$, $s_i'$ must have support in the largest and smallest weights of $\Box_{h_i'}$, which we call $u_i^{\max}$ and $u_i^{\min}$, respectively. Furthermore, by adjusting with $T$-invariant principal divisors we can assume that $(f)=0$,  $u_i^{\min}=0$, and ${h_i'}^*(0)=0$. We then have $({{h'}_i})_{Q_1}^*(u_i^{\max})<2$ for $i=1,2$. Indeed, we must have 
\begin{equation*}
({{h'}_1})_{Q_1}^*(u_1)+ ({{h'}_2})_{Q_1}^*(u_2)< 2
\end{equation*}
 for $u_1\in \Box_{h_1'}$ and $u_2\in \Box_{h_2'}\setminus \{u_2^{\max}\}$. The claim follows for $i=1$ by setting $u_2=0$;  for $i=2$ we just switch the indices. Now,
for at least one $i\in{1,2}$ we must also have $({{h'}_i})_{Q_2}^*(u_i^{\max})<0$. Indeed, this follows from
\begin{equation*}
({{h'}_1})_{Q_2}^*(u_1^{\max})+({{h'}_2})_{Q_2}^*(u_2^{\max})\leq -1.
\end{equation*}
For this $i$, 
\begin{equation*}
	L(D_{h'_i})_{u_i^{\max}}=\Gamma\left(Y,\CO({h'_i}^*(u_i^{\max}))\right)\subset\Gamma\left(Y,\CO(Q_1-Q_2))\right)=0.
\end{equation*}
This is however impossible since we had already concluded that $s'_i$ has support in weight $u_i^{\max}$. 

We have thus shown that a section $s\in L(D_h)$ with $\lambda=0$ is either irreducible, in which case we can bound the number of rational points on it using the Hasse-Weil bound, or it can be decomposed into $T$-invariant components and some remaining section, which either is irreducible or which has support in weights differing by at most 3. Thus, if we  require that $q\geq 89$ and $l\geq q-1$, we have that for any section $s\in L(D_h)$, $Z(s)\leq 3l$. Since our upper bound already states that $d\leq l(q-1)-3l$, we get that in fact
\begin{equation*}
d=l(q-1)-3l.
\end{equation*}

This marks an improvement over the estimates for any of the $T$-codes considered in the previous example. Indeed, to get the desired estimated minimum distance we would have to require $b=0$ and $a\leq3$. Using equation \eqref{eqn:ellipticdim}, one easily checks that the dimension of the resulting code is smaller than $8$.  

\subsection{A Computational Example}
  We are able to provide a $T$-code over $\FF_7$ with parameters $[66,19,30]$, which is as good as the best known code (c.f. \cite{Grassl:codetables}). We set 
  $Y=V(zy^2 + 6x^3 + 4z^3)\subset \PP^2_{\FF_7}$ and consider the divisorial polytope given in figure~\ref{fig:record}. Fixing two $\mathbb{F}_q$-ration points $Q_1$ and $Q_2$ we can compute a generator matrix of $\mathcal{C}(Y,h^*)$ using Macaulay 2 \cite{M2} and the  \verb+toriccodes+ package \cite{ilten:toriccodes}. We can then compute the minimal distance using Magma \cite{MR1484478}.
  
  \begin{figure}[htbp]
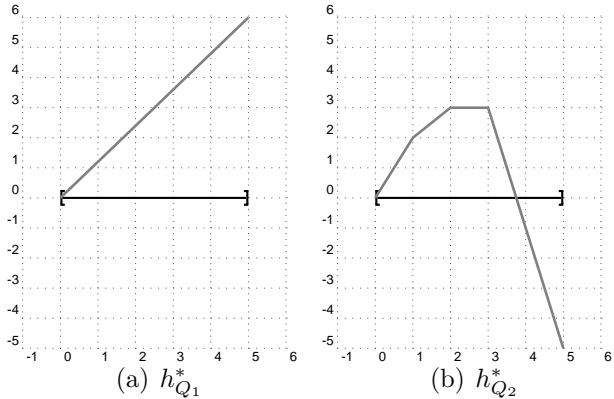

    \centering
    \subfigure[$h^*_{Q_1}$]{\recordi}\hspace{1em}
    \subfigure[$h^*_{Q_2}$]{\recordii}
    \caption{A divisorial polytope defining a $[66,19,30]_7$ code}
   \label{fig:record}
   \end{figure}

   It is easy to see that the length and dimension of $\mathcal{C}(Y,h^*)$ are always respectively $66$ and $19$. However, the minimum distance can be either $29$ or $30$, depending on the choice of $Q_1$ and $Q_2$. For example, setting $Q_1=(1:2:1)$, $Q_2=(1:5:1)$ results in a minimum distance of 30, whereas $Q_1=(1:2:1)$, $Q_2=(0:1:1)$ results in a minimum distance of 29. In fact, the automorphism group of $Y$ divides the set of all pairs of rational points on $Y$ into two equally large subsets; using pairs in one subset results in a minimum distance of $30$, whereas pairs from the other subset result in a minimum distance of $29$. 

   For this example, we are also able to use proposition \ref{prop:kupper} to easily show that $d\leq 30$. Indeed, it is not difficult to find a section $f\in \Gamma(Y,\CO(3Q_1+3Q_2))$ vanishing at $6$ distinct points of $Y(\mathbb{F}_q)\setminus\{Q_1,Q_2\}$. Thus, $f\in L(D_h)_3$ and we get $d\leq 66-6\cdot 6=30$.

\bibliographystyle{alpha}
\bibliography{pdivcodes-arxiv.bib}

\begin{thebibliography}{PHB98}

\bibitem[AH06]{MR2207875}
Klaus Altmann and J{\"u}rgen Hausen.
\newblock Polyhedral divisors and algebraic torus actions.
\newblock {\em Math. Ann.}, 334(3):557--607, 2006.

\bibitem[AHS08]{divfans}
Klaus Altmann, J{\"u}rgen Hausen, and Hendrik S{\"u}{\ss}.
\newblock Gluing affine torus actions via divisorial fans.
\newblock {\em Transformation Groups}, 13(2):215--242, 2008.

\bibitem[AP96]{MR1394921}
Yves Aubry and Marc Perret.
\newblock A {W}eil theorem for singular curves.
\newblock In {\em Arithmetic, geometry and coding theory (Luminy, 1993)}, pages
  1--7. de Gruyter, Berlin, 1996.

\bibitem[BCP97]{MR1484478}
Wieb Bosma, John Cannon, and Catherine Playoust.
\newblock The {M}agma algebra system. {I}. {T}he user language.
\newblock {\em J. Symbolic Comput.}, 24(3-4):235--265, 1997.
\newblock Computational algebra and number theory (London, 1993).

\bibitem[Dem01]{MR1919457}
Jean-Pierre Demailly.
\newblock Multiplier ideal sheaves and analytic methods in algebraic geometry.
\newblock In {\em School on Vanishing Theorems and Effective Results in
  Algebraic Geometry (Trieste, 2000)}, volume~6 of {\em ICTP Lect. Notes},
  pages 1--148. Abdus Salam Int. Cent. Theoret. Phys., Trieste, 2001.

\bibitem[Gra07]{Grassl:codetables}
Markus Grassl.
\newblock {Bounds on the minimum distance of linear codes}.
\newblock Available at \url{http://www.codetables.de}, 2007.
\newblock Accessed on 2008-09-12.

\bibitem[GS08]{M2}
Daniel~R. Grayson and Michael~E. Stillman.
\newblock Macaulay 2, a software system for research in algebraic geometry.
\newblock Available at \url{http://www.math.uiuc.edu/Macaulay2/}, 2008.

\bibitem[Han01]{MR1866342}
S{\o}ren~Have Hansen.
\newblock Error-correcting codes from higher-dimensional varieties.
\newblock {\em Finite Fields Appl.}, 7(4):531--552, 2001.

\bibitem[Han02]{MR1953195}
Johan~P. Hansen.
\newblock Toric varieties {H}irzebruch surfaces and error-correcting codes.
\newblock {\em Appl. Algebra Engrg. Comm. Comput.}, 13(4):289--300, 2002.

\bibitem[Har77]{MR0463157}
Robin Hartshorne.
\newblock {\em Algebraic geometry}.
\newblock Springer-Verlag, New York, 1977.
\newblock Graduate Texts in Mathematics, No. 52.

\bibitem[Ilt08]{ilten:toriccodes}
Nathan Ilten.
\newblock \verb+toriccodes+, a {M}acaulay 2 package for toric- and {T}-codes.
\newblock Available at \url{http://people.cs.uchicago.edu/~nilten/m2.html},
  2008.

\bibitem[Lom03]{lomont:2003}
Chris Lomont.
\newblock {\em Error Correcting Codes on Algebraic Surfaces}.
\newblock PhD thesis, Purdue University, 2003.
\newblock arXiv:math/0309123v1.

\bibitem[LS06]{MR2272243}
John Little and Hal Schenck.
\newblock Toric surface codes and {M}inkowski sums.
\newblock {\em SIAM J. Discrete Math.}, 20(4):999--1014 (electronic), 2006.

\bibitem[PHB98]{MR1667936}
V.~S. Pless, W.~C. Huffman, and R.~A. Brualdi, editors.
\newblock {\em Handbook of coding theory. {V}ol. {I}, {II}}.
\newblock North-Holland, Amsterdam, 1998.

\bibitem[PS08]{petersen-suess08}
Lars Petersen and Hendrik S{\"u}{\ss}.
\newblock Torus invariant divisors.
\newblock arXiv:math/0811.0517v1, 2008.

\bibitem[Rua07]{MR2360532}
Diego Ruano.
\newblock On the parameters of {$r$}-dimensional toric codes.
\newblock {\em Finite Fields Appl.}, 13(4):962--976, 2007.

\bibitem[SS08]{soprunov:2008}
Ivan Soprunov and Jenya Soprunova.
\newblock Toric surface codes and {M}inkowski length of polygons.
\newblock arXiv:0802.2088v1, 2008.

\bibitem[S{\"u}{\ss}08]{suess08}
Hendrik S{\"u}{\ss}.
\newblock Canonical divisors on {T}-varieties.
\newblock arXiv:math/0811.0626v1, 2008.

\end{thebibliography}
\address{
Nathan Ilten\\
Mathematisches Institut\\
Freie Universit\"at Berlin\\
Arnimallee 3\\
14195 Berlin, Germany}{nilten@cs.uchicago.edu}

\address{Hendrik S\"u\ss{}\\
	Institut f\"ur Mathematik\\
        LS Algebra und Geometrie\\
        Brandenburgische Technische Universit\"at Cottbus\\
        PF 10 13 44\\
        03013 Cottbus, Germany}{suess@math.tu-cottbus.de}

\end{document}